\documentclass{article}
\usepackage{graphicx}
\usepackage{amsmath,amsfonts, amsthm, amssymb}
\usepackage[a4paper]{geometry}
\usepackage[colorlinks]{hyperref}
\usepackage{graphicx}
\usepackage{cite}
\usepackage[ruled]{algorithm2e}
\usepackage{array}
\usepackage{float}
\usepackage[caption=false,font=normalsize,labelfont=sf,textfont=sf]{subfig}
\usepackage{textcomp}
\usepackage{stfloats}
\usepackage{url}
\usepackage{verbatim}

\newtheorem{theorem}{Theorem}[section]

\newtheorem{prop}[theorem]{Proposition}

\theoremstyle{definition}
\newtheorem{definition}{Definition}[section]
\newtheorem{example}{Example}[section]

\title{Symmetric Sudoku-Type Games from Perfect Codes}
\author{Junmin An\thanks{junmin0518@sogang.ac.kr, Department of Mathematics and Institute for Mathematical and Data Sciences, Sogang University, Seoul, Korea}, Jae-Hyun Baek\thanks{jhbaek@sogang.ac.kr, Department of Mathematics and Institute for Mathematical and Data Sciences, Sogang University, Seoul, Korea}, Keon-Hwi Kim\thanks{keonhwikim@sogang.ac.kr, Department of Mathematics and Institute for Mathematical and Data Sciences, Sogang University, Seoul, Korea}, Haeun Lim\thanks{haeunlim@sogang.ac.kr, Department of Mathematics and Institute for Mathematical and Data Sciences, Sogang University, Seoul, Korea}, Jon-Lark Kim\thanks{corresponding author, jlkim@sogang.ac.kr, Department of Mathematics and Institute for Mathematical and Data Sciences, Sogang University, Seoul, Korea}}
\date{}

\begin{document}
\maketitle

\begin{abstract}
This paper presents a novel construction method for symmetric Sudoku-type games based on Lee distance perfect codes and diameter perfect codes. The proposed method utilizes the tiling property of these codes to define the structure of the subgrid constraints of Sudoku-type games. In this way, our games inherit the symmetric properties of Sudoku. We provide a detailed analysis of two small cases: a $5 \times 5$ Sudoku in $\mathbb{Z}_5^2$, and an $8 \times 8$ Sudoku in $\mathbb{Z}_8^2$. By defining equivalence relations via rigid motions, we provide a complete enumeration of valid grids, identifying 17 inequivalent solutions for $5\times 5$ Sudoku. For two different types of $8\times 8$ Sudoku, we characterize 232,735 and 304,014 inequivalent solutions, respectively. Furthermore, to verify practical playability, we implement a human-like solver that assesses the difficulty of the generated games. The analysis confirms that our $5\times5$ Sudoku games offer a balanced distribution of difficulty levels, ranging from Easy to Hard, making them a viable alternative to traditional $9 \times 9$ Sudoku.
\end{abstract}

\section{Introduction}
Sudoku~\cite{B07,D06} is a widely recognized logic-based game that was first designed by Howard Garns in 1979 under the name \textit{Number Place}, and published in Dell Magazines. The game was later introduced in Japan by the game company Nikoli in 1984, under the name \textit{Sudoku} which is an abbreviation of the Japanese phrase `s\={u}ji wa dokushin ni kagiru', meaning `the digits are limited to one occurrence'. With its elegant rules and broad appeal, Sudoku quickly gained popularity in Japan and was eventually reintroduced to Western audiences in the early 2000s, sparking a global phenomenon. Since then, Sudoku has achieved worldwide popularity and has now become a recreational pastime enjoyed by people of all ages across cultures.

Moreover, Sudoku has attracted significant attention within the field of mathematics, particularly in combinatorics. The game's inherent structural constraints make it an interesting object of theoretical investigation. These features have inspired a wide range of mathematical studies including the enumeration of valid Sudoku grids~\cite{FJ06}, the classification of Sudoku grids under various equivalence relations~\cite{RJ06}, and the analysis of the algorithmic complexity involved in solving Sudoku games~\cite{YS03}. In addition, many studies have investigated the connections between Sudoku and combinatorial structures such as graphs and Latin squares~\cite{BCC08, KSU23}.

Beyond the standard $9 \times 9$ Sudoku, numerous variants have been devised that alter the size of the grid or the structure of the subblocks. One example is the jigsaw Sudoku, also known as irregular Sudoku, where the grid is partitioned into nine irregularly shaped regions instead of regular $3 \times 3$ blocks. While such variants offer challenging new games, they often lack the rotational and reflective symmetries inherent in conventional Sudoku. Consequently, many Sudoku game solving strategies and theoretical tools developed for standard Sudoku become inapplicable in these cases.

To address this limitation, we aim to propose a generalized framework for constructing Sudoku-type games that preserve high levels of symmetry, analogous to those of the classical Sudoku. Our approach is based on the use of perfect codes with respect to the Lee distance, which naturally yield symmetric, well-structured partitioning of the grid. Beyond their fundamental role in coding theory, perfect codes are also deeply connected to various mathematical structures~\cite{E22}. For example, the structure of perfect codes often aligns with combinatorial designs including symmetric block designs and Steiner systems~\cite{CL75}. Additionally, they have intriguing connections to tiling problems~\cite{EV98, K23}, where optimal configurations under local constraints parallel the decoding spheres of codes. These interconnections have positioned perfect codes as a bridge between coding theory and diverse mathematical disciplines such as design theory, graph theory, and discrete geometry~\cite{CL75}.

In this paper, we further explore the connection between perfect codes and the structure of Sudoku games. We construct several variations of Sudoku games derived from specific classes of perfect codes under the Lee distance, and analyze the structural properties exhibited by each construction.
We provide a detailed analysis of two minimal cases: a $5 \times 5$ Sudoku constructed from a Lee distance perfect code in $\mathbb{Z}_5^2$, and an $8 \times 8$ Sudoku constructed from a Lee distance diameter perfect code in $\mathbb{Z}_8^2$. For each case, we enumerate all valid grids, classify them up to equivalence relations, and consider minimal Sudoku.

In addition, we analyze the difficulty of the Sudoku games constructed from perfect codes. By quantifying their solving complexity using a human-like solver, we demonstrate that these games offer diverse challenges. This analysis confirms that the proposed Sudoku-type games possess not only a rich mathematical structure but also practical value as engaging games. To further highlight this, playable versions of the constructed games have been made available at~\cite{KB25}.

This paper consists of eight sections. We review related works in Section 2. Section 3 introduces preliminaries from coding theory, and Section 4 presents the construction of Sudoku-type games from perfect and diameter perfect codes. In Section 5, we study equivalence relations for the resulting Sudoku grids. In Section 6, we enumerate possible Sudoku grids and the number of minimal Sudoku for small scale perfect and diameter perfect Sudoku grids. In Section 7, we assign difficulty levels to the constructed Sudoku games using a human-like solver. We conclude this paper in Section 8.

\section{Related Work}

The enumeration and classification of Sudoku solutions constitute fundamental analyses in the study of Sudoku as a combinatorial object. Felgenhauer and Jarvis~\cite{FJ06} performed the first complete enumeration of all valid $9 \times 9$ Sudoku solution grids, obtaining approximately $6.671 \times 10^{21}$ distinct solutions, and further classified them into $5{,}472{,}730{,}538$ essentially different solutions via Burnside's lemma under the natural symmetry group. McGuire, Tugemann, and Civario~\cite{MTC14} resolved the minimum number of clues problem by proving that exactly $17$ clues are required to uniquely determine a standard Sudoku solution, through an exhaustive search formulated as a hitting set enumeration. These two results---complete enumeration with equivalence classification and the determination of minimal puzzles---represent standard analytical frameworks that are naturally expected for any well-defined Sudoku-type puzzle. In this paper, we carry out the analogous analyses for the new Sudoku constructed from Lee distance perfect codes.

Assessing the difficulty of a newly constructed Sudoku-type game is essential for evaluating its viability as a playable game. Pel\'{a}nek~\cite{Pel14} proposed a difficulty rating model for standard Sudoku based on a hierarchy of human solving techniques, from basic methods such as Naked Single and Hidden Single to advanced strategies such as Naked Pair and X-Wing. In contrast, our human-like solver deliberately uses only the most elementary deductions---Naked Single and Hidden Single---combined with a trial-and-error phase, measuring difficulty simply by the total number of insertions. This minimal design captures the effort of a general solver relying solely on basic logic, without assuming familiarity with advanced techniques.

\section{Preliminaries on coding theory}
Coding theory traces its origins to Claude Shannon's pioneering paper, `A Mathematical Theory of Communication'~\cite{S48}, which laid the foundation for the mathematical study of reliable information transmission over noisy channels~\cite{HP10}. This groundbreaking insight initiated the study of error-correcting codes, leading to the development of numerous families of codes with various properties. We first review some essential mathematical notions from coding theory. These preliminaries provide the necessary foundation for the discussions that follow.

Let $\mathbb{Z}_n$ be the ring of integers modulo $n$. A \textit{code} $\mathcal{C}$  of length $m$ over $\mathbb{Z}_n$ is a subset of $\mathbb{Z}_n^m$, and its elements are called \textit{codewords}. If $\mathcal{C}$ is a submodule of $\mathbb{Z}_n^m$, then we call $\mathcal{C}$ a \textit{linear code}. A \textit{generator matrix} $G$ for a linear code $\mathcal{C}$ over $\mathbb{Z}_n$ is a matrix whose rows form a minimal spanning set for $\mathcal{C}$. If $n$ is a prime number $p$, then $\mathbb{Z}_p^m$ is a vector space over a finite field $\mathbb{F}_p$, and we can define the dimension of a linear code $\mathcal{C}$.

The \textit{Lee weight} of an element $\mathbf{u}=(u_1, u_2, \cdots, u_m)\in\mathbb{Z}_n^m$ is defined as
\[
\textup{wt}_L(\mathbf{u})=\sum_{i=1}^m\min\{u_i,\, n-u_i\}.
\]
The \textit{Lee distance} between two elements $\mathbf{u}$ and $\mathbf{v}$ in $\mathbb{Z}_n^m$ is given by
$
d_{L}(\mathbf{u}, \mathbf{v})=\textup{wt}_L(\mathbf{u}-\mathbf{v}).
$
Unless otherwise stated, all distances in this paper are Lee distances.

The \textit{minimum distance} $d$ of a code $\mathcal{C}$ over $\mathbb{Z}_n$ is the minimum of the Lee distances between any two distinct codewords in $\mathcal{C}$. A code $\mathcal{C}$ is said to be a \textit{$t$-error-correcting code} if $d\ge 2t+1$. The \textit{minimum weight} of $\mathcal{C}$ is the minimum of the Lee weights of all nonzero codewords in $\mathcal{C}$. In the case of linear codes, the minimum distance coincides with the minimum weight. We denote a code over $\mathbb{Z}_n$ of length $m$, size $M$, and minimum distance $d$ by an $(m, M, d)$ code over $\mathbb{Z}_n$.

Two codes $\mathcal{C}_1$ and $\mathcal{C}_2$ in $\mathbb{Z}_n^m$ are said to be \textit{equivalent} if there exists a permutation matrix $P$ such that
\[
\mathcal{C}_1P=\{\textbf{c}P~|~\mathbf{c}\in\mathcal{C}_1\}=\mathcal{C}_2.
\]
For a code $\mathcal{C}\subseteq\mathbb{Z}_n^m$, a \textit{translate of $\mathcal{C}$ by $\mathbf{x}\in\mathbb{Z}_n^m$} is defined as
\[
\mathbf{x}+\mathcal{C}=\{\mathbf{x}+\mathbf{c}~|~\mathbf{c}\in\mathcal{C}\}.
\]

For a codeword $\mathbf{c}$ in $\mathcal{C}$, we denote by $\mathcal{B}_t(\mathbf{c})$ the ball of radius $t$ centered at $\mathbf{c}$. A code $\mathcal{C}\subseteq\mathbb{Z}_n^m$ is called a \textit{perfect code} if every element of $\mathbb{Z}_n^m$ lies within a ball of radius $t=\lfloor\frac{d-1}{2}\rfloor$ centered at some codeword in $\mathcal{C}$, where $d$ is the minimum distance of the code $\mathcal{C}$. Since the balls of radius $t$ centered at distinct codewords are pairwise disjoint, a perfect code partitions the entire space $\mathbb{Z}_n^m$ into such balls.

Since the size of a ball depends only on its radius and not on its center, we denote by $\mathcal{B}_t$ the ball of radius $t$ in $\mathbb{Z}_n^m$. The following is a classical result, commonly referred to as the \textit{Sphere-Packing Bound}.

\begin{theorem}[Sphere-Packing Bound]
Let $\mathcal{C}\subseteq \mathbb{Z}_n^m$ be a code with minimum distance $d$, and let $t=\lfloor\frac{d-1}{2}\rfloor$. Then
\[
|\mathcal{C}|\cdot|\mathcal{B}_t|\le |\mathbb{Z}_n^m|.
\]
\end{theorem}
A code that attains equality in this bound is also called a perfect code.

To consider a generalization of perfect codes, we introduce the notion of an \textit{anticode}. An anticode $\mathcal{A}\subseteq \mathbb{Z}_n^m$ with diameter $D$ is a subset of $\mathbb{Z}_n^m$ in which the maximum of the distances between any two elements of $\mathcal{A}$ is $D$. This notion leads to a generalization of the Sphere-Packing Bound, known as the \textit{Code-Anticode Bound}~\cite{D73}.
\begin{theorem}[Code-Anticode Bound]\label{cabound}
	Let $\mathcal{C}\subseteq\mathbb{Z}_n^m$ be a code with minimum distance $d$, and let $\mathcal{A}$ be an anticode with diameter $D=d-1$. Then
	\[
	|\mathcal{C}|\cdot |\mathcal{A}|\le |\mathbb{Z}_n^m|.
	\]
\end{theorem}
A code $\mathcal{C}$ that attains equality in this bound is called a \textit{$D$-diameter perfect code}.

For an adjacent pair of points $\{p_1, p_2\}$ in $\mathbb{Z}_n^m$, and for a positive integer $t$ with $2t+1\le n$, we define $\mathcal{A}_{2t+1}$ to be the set of all points in $\mathbb{Z}_n^m$ whose distance from $\{p_1, p_2\}$ is at most $t$. We call $\{p_1, p_2\}$ the \textit{core} of $\mathcal{A}_{2t+1}$.

\begin{theorem}{\rm(\cite{E22})}\label{anticodesize}
The set $\mathcal{A}_{2t+1}\subseteq\mathbb{Z}_n^m$ is an anticode of maximum size with diameter $2t+1$ where $2t+1\le n$, and the ball $\mathcal{B}_t$ is an anticode of maximum size with diameter $2t$ where $2t\le n$. In particular, their sizes are given as follows:
\begin{enumerate}
\item[(i)] $\displaystyle |\mathcal{A}_{2t+1}|=\sum_{i=0}^{\min\{m-1, t\}}2^{i+1}{\binom{m-1}{i}}{\binom{t+1}{i+1}}$.
\item[(ii)] $\displaystyle |\mathcal{B}_t|=\sum_{i=0}^{\min\{m, t\}}2^{i}{\binom{m}{i}}{\binom{t}{i}}$.
\end{enumerate}
\end{theorem}

Let $\mathcal{C}$ be a code with minimum distance $d$. If $\mathcal{A}$ is an anticode of diameter $d-1$, then each translate of $\mathcal{A}$ contains at most one codeword of $\mathcal{C}$. We also have the following theorem.
\begin{theorem}{\rm(\cite{E22})}\label{dperfectcode}
Let $\mathcal{C}\subseteq\mathbb{Z}_n^m$ be a code with minimum distance $d$ and let $\mathcal{A}$ be a corresponding anticode with diameter $D=d-1$. Then the code $\mathcal{C}$ is a $D$-diameter perfect code and $\mathcal{A}$ is a maximum size anticode if and only if each translate of $\mathcal{A}$ contains exactly one codeword of $\mathcal{C}$.
\end{theorem}

The following two theorems characterize perfect codes and diameter perfect codes of length $2$.

\begin{theorem}\label{pcode}{\rm(\cite{GW70})}
Let $t\ge 1$, and let $\mathcal{C}\subseteq\mathbb{Z}_n^2$ with $n=2t^2+2t+1$ be a linear code with generator matrix
\[
G=\begin{bmatrix} 1 & 2t+1\end{bmatrix}.
\]
Then $\mathcal{C}$ is a $t$-error-correcting perfect code. Moreover, every $t$-error-correcting perfect code $\mathcal{C}\subseteq\mathbb{Z}_n^2$ is a translate of $\mathcal{C}$ or a translate of a code equivalent to $\mathcal{C}$.
\end{theorem}

\begin{theorem}{\rm{(\cite{E22})}}\label{leeanticode}
For each $0\le i\le t$, the following matrix
\[
G_i=\begin{bmatrix}
    t+1+i & t+1-i\\i & 2(t+1)-i
\end{bmatrix}
\]
generates a $(2t+1)$-diameter perfect code $\mathcal{C}\subseteq\mathbb{Z}_n^2$, where $n=2(t+1)^2$, with minimum distance $2t+2$ and an anticode $\mathcal{A}_{2t+1}$. In particular, for each $1\le i\le t$, $G_0$ and $G_i$ are not equivalent to each other.
\end{theorem}

\section{Construction of Sudoku-type games from perfect codes}
A Sudoku game is a $9 \times 9$ grid, partially filled with digits from 1 to 9. The grid is further divided into nine non-overlapping $3 \times 3$ subgrids. The objective is to complete the grid by filling the remaining entries so that each digit from 1 to 9 appears exactly once in every row, column, and $3\times 3$ subgrid.

To formalize the structure of a Sudoku game, we express it in terms of a Latin square, which captures the row and column constraints inherent to the game. A \textit{Latin square} of order $n$ is an $n \times n$ array whose entries are filled with elements from a set of size $n$, such that each element appears exactly once in every row and column.

Two Latin squares of order $n$ are said to be \textit{orthogonal} to each other if the set of ordered pairs of corresponding entries includes each of the $n^2$ possible pairs exactly once. In this paper, we use the term \textit{orthogonal} more broadly to refer to any two $n \times n$ arrays, not necessarily Latin squares, with entries from a set of size $n$ such that the set of ordered pairs of their corresponding entries contains all $n^2$ possible pairs exactly once.

\begin{definition}
Let $n$ be a positive integer and $[n]=\{1, 2, \cdots, n\}$. An $n\times n$ array $\mathcal{I}$ with entries from $[n]$ is called a \textit{palette grid} if each element $k$ in $[n]$ appears exactly $n$ times in $\mathcal{I}$, that is,
\[
|\{(i, j) \in \mathbb{Z}_n^2 \mid \mathcal{I}_{ij} = k\}| = n\quad\mbox{for each }k\in[n].
\]
An $n \times n$ array $S_\mathcal{I}$ over the symbol set $[n]$ is called a \textit{Sudoku grid} with respect to $\mathcal{I}$ if the following conditions are satisfied:
\begin{enumerate}
\item[(i)] $S_\mathcal{I}$ is a Latin square of order $n$.
\item[(ii)] $S_\mathcal{I}$ is orthogonal to the palette grid $\mathcal{I}$.
\end{enumerate}
\end{definition}

To avoid confusion, a Sudoku game is referred to as Sudoku, and a fully completed array satisfying all constraints is referred to as a Sudoku grid.

For the standard $9 \times 9$ Sudoku, the palette grid $\mathcal{I}$ would look like:
{\small{
\begin{equation}\label{Sudoku}
\mathcal{I}  = \begin{bmatrix}
1 & 1 & 1 & 2 & 2 & 2 & 3 & 3 & 3 \\
1 & 1 & 1 & 2 & 2 & 2 & 3 & 3 & 3 \\
1 & 1 & 1 & 2 & 2 & 2 & 3 & 3 & 3 \\
4 & 4 & 4 & 5 & 5 & 5 & 6 & 6 & 6 \\
4 & 4 & 4 & 5 & 5 & 5 & 6 & 6 & 6 \\
4 & 4 & 4 & 5 & 5 & 5 & 6 & 6 & 6 \\
7 & 7 & 7 & 8 & 8 & 8 & 9 & 9 & 9 \\
7 & 7 & 7 & 8 & 8 & 8 & 9 & 9 & 9 \\
7 & 7 & 7 & 8 & 8 & 8 & 9 & 9 & 9
\end{bmatrix}.
\end{equation}
}}
\\
Condition (ii), that is $\mathcal{I}\perp S_\mathcal{I}$, enforces the constraint that within each $3\times 3$ subgrid of $S_\mathcal{I}$, every symbol $1$ to $9$ must appear exactly once. In this sense, the orthogonality to the palette grid $\mathcal{I}$ represents the subgrid condition of Sudoku.

From this perspective, constructing novel Sudoku-type games is equivalent to constructing suitable $n\times n$ palette grid $\mathcal{I}$. For example, the following $9\times 9$ array
{\small{
\[
\mathcal{I}_{j}=\begin{bmatrix}
1 & 1 & 1 & 1 & 2 & 2 & 3 & 3 & 3 \\
1 & 1 & 1 & 2 & 2 & 3 & 3 & 3 & 3 \\
1 & 4 & 2 & 2 & 5 & 3 & 6 & 3 & 6 \\
1 & 4 & 2 & 5 & 5 & 5 & 6 & 6 & 6 \\
4 & 4 & 2 & 2 & 5 & 5 & 5 & 5 & 6 \\
4 & 4 & 7 & 7 & 7 & 7 & 9 & 5 & 6 \\
4 & 4 & 7 & 8 & 7 & 7 & 9 & 6 & 6 \\
4 & 8 & 7 & 8 & 8 & 7 & 9 & 9 & 9 \\
8 & 8 & 8 & 8 & 8 & 9 & 9 & 9 & 9
\end{bmatrix}
\]
}}
defines the palette grid for the jigsaw Sudoku grids introduced in~\cite{L18}. Unlike the jigsaw Sudoku grid with respect to $\mathcal{I}_{j}$, the conventional Sudoku grid, whose palette grid is given in~(\ref{Sudoku}), exhibits a rich underlying symmetric structure. By constructing palette grids using other mathematical objects with high symmetry, one can generate new Sudoku-type games similar to those of conventional Sudoku. In particular, perfect codes under Lee distance can serve as a foundational mathematical object for generating symmetric palette grids of this kind.

\subsection{Sudoku from perfect codes}
Let $n=2t^2+2t+1$ for some positive integer $t$, and consider $\mathbb{Z}_n$, the ring of integers modulo $n$. Let $\mathcal{C}\subseteq\mathbb{Z}_n^2$ be a perfect code with minimum distance $d=2t+1$. Then the balls of radius $t=(d-1)/2$ centered at codewords of $\mathcal{C}$ are pairwise disjoint and satisfy
\[
\bigcup_{\mathbf{c}\in \mathcal{C}}\mathcal{B}_t(\mathbf{c})=\mathbb{Z}_n^2.
\]
Moreover, the size of each ball is $|\mathcal{B}_t(c)|=n$ because $4\cdot \frac{t(t+1)}{2}+1=n$. Throughout this subsection, we fix this value of $n$ as $2t^2 + 2t + 1$.

\begin{definition}
Let $\mathcal{C} = \{\mathbf{c}_1, \mathbf{c}_2, \dots, \mathbf{c}_n\} \subseteq \mathbb{Z}_n^2$ be a $t$-error-correcting perfect code. The palette grid $\mathcal{I}_\mathcal{C}$ with respect to $\mathcal{C}$ is defined as an $n \times n$ array over the symbol set $[n]$, where rows and columns are indexed by the elements of $\mathbb{Z}_n$ in increasing order $0, 1, \dots, n-1$, such that for each $1 \le i \le n$ and every $(x, y) \in \mathcal{B}_t(\mathbf{c}_i)$, the entry at position $(x, y)$ is set to be $i$. An $n\times n$ array $S$ over $[n]$ is called a \textit{perfect Sudoku grid} with respect to $\mathcal{C}$ if it is a Sudoku grid with respect to the palette grid $\mathcal{I}_\mathcal{C}$.
\end{definition}

\begin{example}\label{dpcodeex}
Let $t=1$ and $n=2t^2+2t+1=5$. Let $\mathcal{C}\subseteq\mathbb{Z}_5^2$ be a code with generator matrix $G=\begin{bmatrix}3 & 1\end{bmatrix}$, so that
\[
\mathcal{C}=\{(0,\, 0),\,(3,\,1),\,(1,\,2),\,(4,\,3),\,(2,\,4)\}.
\]
It forms a perfect code over $\mathbb{Z}_5$ with minimum distance $3$. The palette grid $\mathcal{I}_\mathcal{C}$ is given as follows:
\[
\mathcal{I}_\mathcal{C}=\begin{bmatrix}
    \mathbf{1} & 1 & 2 & 5 & 1\\1 & 2 & \mathbf{2} & 2 & 3\\3 & 4 & 2 & 3 & \mathbf{3}\\4 & \mathbf{4} & 4 & 5 & 3\\1 & 4 & 5 & \mathbf{5} & 5
\end{bmatrix},
\]
where the positions of boldfaces correspond to the codewords of $\mathcal{C}$. Let
\[
S=\begin{bmatrix}
    4 & 5 & 1 & 2 & 3\\1 & 2 & 3 & 4 & 5\\3 & 4 & 5 & 1 & 2\\5 & 1 & 2 & 3 & 4\\2 & 3 & 4 & 5 & 1
\end{bmatrix}
\]
be a Latin square of order $5$. Then $S$ is orthogonal to the palette grid $\mathcal{I}_\mathcal{C}$. Thus, $S$ is a perfect Sudoku grid with respect to $\mathcal{C}$.
\end{example}

\subsection{Sudoku from diameter perfect codes}
Let $\mathcal{C}\subseteq\mathbb{Z}_n^2$ be a code with minimum distance $d$. From Theorem \ref{anticodesize}, we know that the maximum size of an anticode in $\mathbb{Z}_n^2$ with diameter $2t+1<d$ and $t\ge 1$ is
\[
|\mathcal{A}_{2t+1}|=\sum_{i=0}^t2^{i+1}{\binom{1}{i}}{\binom{t+1}{i+1}}=2(t+1)^2.
\]
We now set $n=2(t+1)^2$, and consider a code $\mathcal{C}\subseteq\mathbb{Z}_n^2$. By the Code-Anticode Bound, we have $|\mathcal{C}|=n$. This value of $n=2(t+1)^2$ will be fixed throughout this subsection.

\begin{prop}
Let $\mathcal{C}_i\subseteq\mathbb{Z}_n^2$ where $0\le i\le t$ be a code with generator matrix
\[
G_i=\begin{bmatrix} t+1+i & t+1-i\\i & 2(t+1)-i\end{bmatrix}.
\]
Then, each translate of $\mathcal{C}_i$ is a ($2t+1$)-diameter perfect code with minimum distance $2t+2$ and an anticode $\mathcal{A}_{2t+1}$.
\end{prop}
\begin{proof}
This follows directly from Theorem~\ref{dperfectcode}, which ensures that each translate of the anticode $\mathcal{A}_{2t+1}$ in $\mathbb{Z}_n^2$ contains exactly one codeword of $\mathcal{C}_i$.
\end{proof}

In this section, we assume that for each codeword $\mathbf{c}$ in a $t$-error-correcting diameter perfect code $\mathcal{C} \subseteq \mathbb{Z}_n^2$, the core of the translate of the anticode $\mathcal{A}_{2t+1}$ corresponding to $\mathbf{c}$ is given by the set $\{\mathbf{c}, \mathbf{c}+(1, 0)\}$.

\begin{definition}
    Let $\mathcal{C}=\{\mathbf{c}_1, \mathbf{c}_2, \ldots, \mathbf{c}_n\}\subseteq\mathbb{Z}_n^2$ where $n=2(t+1)^2$ be a $(2t+1)$-diameter perfect code $\mathcal{C}\subseteq\mathbb{Z}_n^2$ with minimum distance $2t+2$ and an anticode $\mathcal{A}_{2t+1}$. Define a palette grid $\mathcal{I}_{\mathcal{C}, \mathcal{A}}$ as an $n\times n$ array over the symbol set $[n]$, where rows and columns are indexed by the elements of $\mathbb{Z}_n$ in increasing order. Let $\mathcal{A}=\{\mathcal{A}_i\}_{i=1}^n$ be the set of translates of $\mathcal{A}_{2t+1}$ where $\mathcal{A}_i$ contains $\mathbf{c}_i$ for each $i$. For each $1\le i\le n$ and every $(x, y)\in\mathcal{A}_i$, set the entry at position $(x, y)$ of $\mathcal{I}_{\mathcal{C}, \mathcal{A}}$ to be $i$. We call an $n\times n$ array $S$ over $[n]$ a \textit{diameter perfect Sudoku grid} with respect to $\mathcal{C}$ and $\mathcal{A}$ if it is a Sudoku grid with respect to the palette grid $\mathcal{I}_{\mathcal{C}, \mathcal{A}}$.
\end{definition}

\begin{example}
    Let $t=1$ and $G_1$ be a generator matrix given in Theorem~\ref{leeanticode}. Then,
    \[
    G_1=\begin{bmatrix}
        3 & 1
    \end{bmatrix}
    \]
    generates a $3$-diameter perfect code $\mathcal{C}\subseteq\mathbb{Z}_8^2$ with minimum distance $4$, which is given by
    \[
    \mathcal{C}=\{(3, 1), (6, 2), (1, 3), (4, 4), (7, 5), (2, 6), (5, 7), (0, 0)\}.
    \]
    For $1\le i\le 8$, we define an anticode $\mathcal{A}_i$ whose core consists of $\{(3i, i), (3i+1, i)\}$, denoted in boldface. Let $\mathcal{A}=\{\mathcal{A}_i\}_{i=1}^8$. Then, we have a palette grid $\mathcal{I}_{\mathcal{C}, \mathcal{A}}$ as follows:
\[\mathcal{I}_{\mathcal{C}, \mathcal{A}}=
\begin{bmatrix}
    \mathbf{8} & 8 & 2 & 3 & 5 & \mathbf{5} & 5 & 8\\\mathbf{8} & 8 & 3 & \mathbf{3} & 3 & 5 & 6 & 8\\8 & 1 & 3 & \mathbf{3} & 3 & 6 & \mathbf{6} & 6\\1 & \mathbf{1} & 1 & 3 & 4 & 6 & \mathbf{6} & 6\\1 & \mathbf{1} & 1 & 4 & \mathbf{4} & 4 & 6 & 7\\7 & 1 & 2 & 4 & \mathbf{4} & 4 & 7& \mathbf{7}\\7 & 2 & \mathbf{2} & 2 & 4 & 5 & 7 & \mathbf{7}\\8 & 2 & \mathbf{2} & 2 & 5 & \mathbf{5} & 5 & 7
\end{bmatrix}.
\]
Let $S$ be a Latin square of order $8$, whose entries are as follows:
\[
S=\begin{bmatrix}
    2 & 4 & 6 & 1 & 5 & 3 & 8 & 7\\
    8 & 5 & 3 & 6 & 4 & 7 & 2 & 1\\
    6 & 1 & 2 & 5 & 7 & 8 & 4 & 3\\
    4 & 2 & 7 & 8 & 3 & 6 & 1 & 5\\
    5 & 3 & 8 & 4 & 2 & 1 & 7 & 6\\
    1 & 6 & 4 & 7 & 8 & 5 & 3 & 2\\
    7 & 8 & 1 & 3 & 6 & 2 & 5 & 4\\
    3 & 7 & 5 & 2 & 1 & 4 & 6 & 8\\
\end{bmatrix}.
\]
It can be verified that $S$ is orthogonal to $\mathcal{I}_{\mathcal{C}, \mathcal{A}}$ and hence $S$ forms a diameter perfect Sudoku grid with respect to $\mathcal{C}$ and $\mathcal{A}$.
\end{example}

\section{Equivalences of Sudoku grids from perfect codes}
Let $\mathcal{S}_\mathcal{C}$ be the set of all perfect Sudoku grids constructed from a $t$-error-correcting perfect code $\mathcal{C}\subseteq\mathbb{Z}_n^2$ and $\mathcal{S}_{\mathcal{C},\mathcal{A}}$ be the set of all diameter perfect Sudoku grids constructed from a ($2t+1$)-diameter perfect code $\mathcal{C}\subseteq \mathbb{Z}_n^2$ with minimum distance $2t+2$ and a corresponding anticode $\mathcal{A}_{2t+1}$. For simplicity, we denote both $\mathcal{S}_\mathcal{C}$ and $\mathcal{S}_{\mathcal{C},\mathcal{A}}$ by $\mathcal{S}$.

We define an equivalence relation $\sim_l$ on $\mathcal{S}$ by declaring $S_1\sim_l S_2$ for $S_1, S_2\in\mathcal{S}$ if there exists a bijection $\sigma:[n]\to[n]$ such that $\sigma((S_1)_{i, j})=(S_2)_{i, j}$ for every $i, j\in\mathbb{Z}_n$. In other words, two (diameter) perfect Sudoku grids are considered equivalent if one can be obtained from the other by applying a permutation of symbols in $[n]$ to each entry. We refer to this equivalence relation as \textit{relabeling}. It is straightforward to verify that $\sim_l$ is an equivalence relation on $\mathcal{S}$. Moreover, for each $S\in \mathcal{S}$, the equivalence class $[S]=\{S'\in \mathcal{S}|S'\sim_l S\}$ contains exactly $n!$ elements. We denote the set of equivalence classes of $\mathcal{S}$ up to relabeling by $\bar{\mathcal{S}}$.

Next, we define an equivalence relation on $\bar{\mathcal{S}}$ by considering the action of the group of rigid motions acting on $n \times n$ (diameter) perfect Sudoku grids. Specifically, we introduce three types of transformations that can be applied to a set $\mathcal{A}$ of $n\times n$ arrays.
\begin{enumerate}
\item[(i)] Rotation: Let $r:\mathcal{A}\to\mathcal{A}$ be the map that rotates an array $A\in\mathcal{A}$ by $90^\circ$ clockwise. Specifically, $r$ is given by
	\[(r(A))_{i,j}=A_{n-1-j, i}\]
	for all $0\le i, j\le n-1$.
\item[(ii)] Reflection: Let $s:\mathcal{A}\to\mathcal{A}$ be the map that reflects an array $A\in\mathcal{A}$ with respect to the vertical axis. The map $s$ is defined by
	\[(s(A))_{i, j}=A_{i, n-1-j}
	\]
	for all $0\le i, j\le n-1$.
\item[(iii)] Translation: Let $\tau_1, \tau_2:\mathcal{A}\to\mathcal{A}$ be the maps that translate an array $A\in\mathcal{A}$ downward by one row and rightward by one column, respectively. Specifically,
	\[(\tau_1(A))_{i, j}=A_{i-1, j}\quad\mbox{and}\quad (\tau_2(A))_{i,j}=A_{i, j-1}\]
	for all $0\le i, j\le n-1$, where the indices are taken modulo $n$.
\end{enumerate}
These transformations form the group $\mathcal{G}$ of rigid motions under composition. Then $\mathcal{G}$ defines an action on the set $\mathcal{A}$ of $n \times n$ arrays, where the action is given by $\varphi \cdot A = \varphi(A)$, and composition satisfies $\varphi_1\varphi_2(A)=\varphi_1(\varphi_2(A))$ for all $A \in \mathcal{A}$ and $\varphi,\, \varphi_1,\, \varphi_2\in \mathcal{G}$.

In order to define a well-defined group action of $\mathcal{G}$ on the set $\bar{\mathcal{S}}$ of relabeling equivalence classes, we must ensure that the action of each group element preserves the set $\mathcal{S}$ of (diameter) perfect Sudoku grids. That is, we seek a subgroup $\mathcal{G}_\mathcal{S} \leq \mathcal{G}$ such that
	$g \cdot S \in \mathcal{S} \quad \text{for all } S \in \mathcal{S} \text{ and } g \in \mathcal{G}_\mathcal{S}.$

\subsection{Sudoku from perfect codes}
\begin{theorem}\label{pSudoku}
Let $\mathcal{C}_0$ be a linear $t$-error-correcting perfect code with generator matrix
\[
G=[a~~b],
\]
and $\mathcal{C}=\mathbf{x}+\mathcal{C}_0\subseteq\mathbb{Z}_n^2$ for some $\mathbf{x}=(x_1, x_2)\in\mathbb{Z}_n^2$.
Define
	\[
	\mathcal{G}_S=\langle \tau_2^{x_1+x_2+1}\tau_1^{x_1-x_2}r,\, \tau_1^a\tau_2^b\rangle.
	\]
	Then, for every $S\in \mathcal{S}$ and $g\in \mathcal{G}_\mathcal{S}$, $g\cdot S\in \mathcal{S}$.
\end{theorem}
\begin{proof}
	We first consider $\tau_1^a\tau_2^b$. For $S\in \mathcal{S}$, $\tau_1^a\tau_2^b(S)$ is a Latin square of order $n$ since translation preserves the property that each row and column contains distinct symbols. Next, we show that $\tau_1^a \tau_2^b(S)$ is orthogonal to $\mathcal{I}_{\mathcal{C}}$. For each $\mathbf{c}\in \mathcal{C}$, note that $\mathbf{c}+(a, b)\in\mathcal{C}$. Then $\mathcal{B}_t(\mathbf{c}+(a, b))=\mathcal{B}_t(\mathbf{c}')$ for some $\mathbf{c}'\in\mathcal{C}$. By the definition of the palette grid $\mathcal{I}_{\mathcal{C}}$, it follows that the translated grid $\tau_1^a \tau_2^b(S)$ is orthogonal to $\mathcal{I}_{\mathcal{C}}$. Therefore, $\tau_1^a \tau_2^b(S)\in \mathcal{S}$.
	
	Next, we consider $\tau_2^{x_1 + x_2 + 1} \tau_1^{x_1 - x_2} r$. Since a rotation also preserves the Latin square property, it suffices to show that the resulting grid is orthogonal to $\mathcal{I}_{\mathcal{C}}$. Let $x'=x_1+x_2+1$. For each $\mathbf{c}=(c_1, c_2)\in\mathcal{C}$, $\mathbf{c}=\mathbf{x}+k(a, b)$ for some $k\in\mathbb{Z}_n$. Then we have
		\begin{align*}
			(\tau_2^{x'}(\tau_1^{x_1 - x_2} r(S)))_{c_1, c_2}&=(\tau_1^{x_1 - x_2} r(S))_{c_1, c_2-x'}\\&=(r(S))_{c_1-(x_1-x_2), c_2-x'}\\&=S_{x_1+x_2-c_2, c_1-x_1+x_2}\\&=S_{x_1-kb, x_2+ka}.
		\end{align*}
		Let $\mathbf{c}''=\mathbf{x}+k(-b, a)$. Consider that $(a, b)=h(1, 2t+1)P$ for some permutation matrix $P$ and some unit $h\in\mathbb{Z}_n$ by Theorem~\ref{pcode}. If $P=I$, then
			\[
			-(2t+1)\cdot(a, b)=-(2t+1)\cdot h(1, 2t+1)=h(-(2t+1), 1)=(-b, a),
			\]
		and if $P=\begin{bmatrix} 0 & 1\\1 & 0\end{bmatrix}$, then
		\[
		(2t+1)\cdot(a, b)=(2t+1)\cdot h(2t+1, 1)=h(-1, 2t+1)=(-b, a),
		\]
        in $\mathbb{Z}_n$ where $n=2t^2+2t+1$.
		Thus $\mathbf{c}''\in\mathcal{C}$, and $\tau_2^{x_1 + x_2 + 1} \tau_1^{x_1 - x_2} r(S)$ is orthogonal to $\mathcal{I}_{\mathcal{C}}$. Hence it belongs to $\mathcal{S}$. This completes the proof.
\end{proof}

The action of $\mathcal{G}_\mathcal{S}$ on a perfect Sudoku grid $S$ with respect to the code $\mathcal{C} = (2, 2) + \mathcal{C}_0$, where $\mathcal{C}_0$ is the code given in Example~\ref{dpcodeex}, is illustrated in Figure~\ref{fig:5x5_example} in which the original grid is mapped to two grids under the action of $\tau_1^3\tau_2$ and $r$.

\begin{figure}[ht]
\centering
\includegraphics[width=0.7\textwidth]{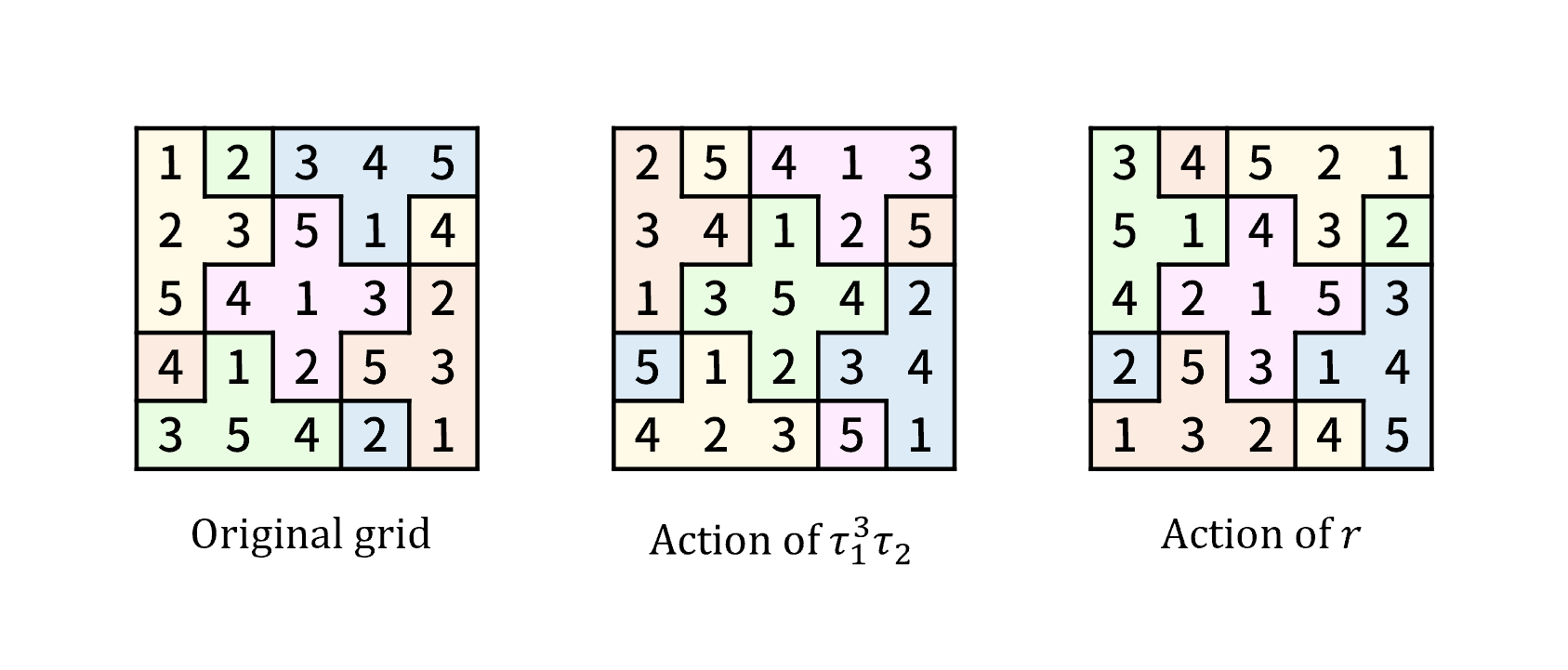}
\caption{Action of $\mathcal{G}_\mathcal{S}$ on a perfect Sudoku grid $S\in\mathcal{S}$ with respect to $\mathcal{C}$}
\label{fig:5x5_example}
\end{figure}

The action of $\mathcal{G}_\mathcal{S}$ on $\mathcal{S}$ naturally induces an action on the set $\bar{\mathcal{S}}$ of relabeling equivalence classes. This, in turn, defines an equivalence relation on $\bar{\mathcal{S}}$, where each equivalence class is an orbit $\mathcal{G}_\mathcal{S}\cdot\bar{S}$ for some $\bar{S}\in \bar{\mathcal{S}}$.

\begin{theorem}
Let $\mathcal{C}_0$ be a linear $t$-error-correcting perfect code with generator matrix
\[
G=[a~~b],
\]
and let $\mathcal{C}=\mathbf{x}+\mathcal{C}_0\subseteq\mathbb{Z}_n^2$ for some $\mathbf{x}=(x_1, x_2)\in\mathbb{Z}_n^2$ with $n=2t^2+2t+1$.
We define the group $\mathcal{G}_\mathcal{S}$ by
\[\mathcal{G}_\mathcal{S}=\langle \tau_2^{x_1+x_2+1}\tau_1^{x_1-x_2}r, \tau_1^a\tau_2^b\rangle.
\]
Then, each equivalence class in $\bar{\mathcal{S}}$ under the action of $\mathcal{G}_{\mathcal{S}}$ has a size which is a divisor of $4n$.
\end{theorem}
\begin{proof}
Note that the rotation $r$ has order $4$ since applying it four times returns the grid to its original configuration, and the translations $\tau_1$ and $\tau_2$, which shift the grid cyclically in the row and column directions, respectively, each has order $n$. Since $(a, b)$ is either $(1,\, 2t+1)$ or $(2t+1,\, 1)$, it follows that the order of $\tau_1^a \tau_2^b$ is $n$.
	
	Let
	\[
	\varphi=\tau_2^{x_1+x_2+1}\tau_1^{x_1-x_2}r.
	\]
	For $A\in\mathcal{A}$ and $i, j\in\mathbb{Z}_n$,
	\begin{align*}
	(\varphi^4(A))_{i, j}&=(\varphi^2(\varphi^2(A)))_{i, j}\\&=(\varphi(\varphi^2(A)))_{x_1+x_2-j, -x_1+x_2+i}\\&=(\varphi^2(A))_{2x_1-i, 2x_2-j}\\&=A_{i, j}.
	\end{align*}
	Therefore, $\varphi^4(A) = A$ for all $A \in \mathcal{A}$, and hence the order of $\varphi$ is $4$. It follows that the subgroup $\mathcal{G}_\mathcal{S} = \langle \varphi,\, \tau_1^a \tau_2^b \rangle$ has order $4n$ since $n$ is odd. By the orbit-stabilizer theorem, the size of each equivalence class in $\bar{\mathcal{S}}$ is a divisor of $4n$.
\end{proof}

\subsection{Sudoku from diameter perfect codes}
We utilize the generator matrices $G_i$ defined in Theorem \ref{leeanticode} to construct diameter perfect codes. We denote the matrices corresponding to the cases $i=0$ and $i=1$ as $G'$ and $G''$, respectively:
\[G' = G_0 = \begin{bmatrix} t+1 & t+1 \\ 0 & 2(t+1) \end{bmatrix}\]
and
\[\quad G'' = G_1 = \begin{bmatrix} t+2 & t \\ 1 & 2t+1 \end{bmatrix}= \begin{bmatrix} 1 & 2t+1 \end{bmatrix}\] since the first row of $G''$ is a multiple of the second row. Let $\mathcal{C}'$ and $\mathcal{C}''$ denote the codes generated by $G'$ and $G''$, respectively.

Let $\mathcal{C}$ be a translate of a linear $t$-error-correcting diameter perfect code $\mathcal{C}_0$. We refer to $\mathcal{C}$ as \textit{Case I} if $\mathcal{C}_0$ is equivalent to $\mathcal{C}'$, and as \textit{Case II} if $\mathcal{C}_0$ is equivalent to $\mathcal{C}''$. When $t=1$, every $1$-error-correcting diameter perfect code is necessarily of either Case~I or Case~II. We restrict our attention to codes that belong to Case I or Case II.

Suppose $\mathcal{C}$ is of Case I, where $\mathcal{C} = \mathbf{x} + \mathcal{C}'P$ for some $\mathbf{x} \in \mathbb{Z}_n^2$ and a permutation matrix $P$. If $P\ne I$, then a generator matrix for $\mathcal{C}_0$ is given by
\[
\begin{bmatrix} t+1 & t+1 \\ 0 & 2(t+1) \end{bmatrix}
\begin{bmatrix} 0 & 1 \\ 1 & 0 \end{bmatrix}=
\begin{bmatrix} t+1 & t+1 \\ 2(t+1) & 0 \end{bmatrix},
\]
which generates the same code as
\[G'=\begin{bmatrix} t+1 & t+1 \\ 0 & 2(t+1) \end{bmatrix}.\]
So, without loss of generality, we assume $\mathcal{C}_0=\mathcal{C}'$.

\begin{theorem}
Let $\mathcal{C}_0$ be a linear $t$-error-correcting diameter perfect code with generator matrix $G_0$ and $\mathcal{C}=\mathbf{x}+\mathcal{C}_0\subseteq\mathbb{Z}_n^2$ for some $\mathbf{x}=(x_1, x_2)\in\mathbb{Z}_n^2$. Suppose $\mathcal{C}$ is either of Case I or Case II.
	\begin{enumerate}
	\item[(i)] If $\mathcal{C}$ is of Case I, let
	\[
	\mathcal{G}_\mathcal{S}=\langle \tau_1^{t+1}\tau_2^{t+1},\, \tau_2^{2(t+1)},\,\tau_2^{2x_2+1}s,\, \tau_1^{2x_1+2}\tau_2^{2x_2+1}r^2\rangle.
	\]
	\item[(ii)] If $\mathcal{C}$ is of Case II and $G_0=[a~~b]$, let
	\[
	\mathcal{G}_\mathcal{S}=\langle \tau_1^{a}\tau_2^{b},\, \tau_1^{2x_1+2}\tau_2^{2x_2+1}r^2\rangle.
	\]
	\end{enumerate}
	Then, for every $S\in \mathcal{S}$ and $g\in \mathcal{G}_\mathcal{S}$, we have $g\cdot S\in \mathcal{S}$.
\end{theorem}
\begin{proof}
We show that, in each case, the action of every generator of $\mathcal{G}_\mathcal{S}$ on a Sudoku grid $S \in \mathcal{S}$ preserves orthogonality with the corresponding palette grid $\mathcal{I}_{\mathcal{C}, \mathcal{A}}$.
\begin{enumerate}
\item[(i)] Let us assume that $\mathcal{C}$ is of Case I. We consider each generator of $\mathcal{G}_\mathcal{S}$ separately. Take any $\mathbf{c}=(c_1, c_2)\in\mathcal{C}$ and $S\in\mathcal{S}$.
	\begin{itemize}
	\item[(1)] We have
		\[
		(\tau_1^{t+1}\tau_2^{t+1}(S))_{c_1, c_2}=S_{c_1-(t+1), c_2-(t+1)}
		\]
		and
		\[
		(\tau_2^{2(t+1)}(S))_{c_1, c_2}=S_{c_1, c_2-2(t+1)}.
		\]
		Since both $(c_1-(t+1), c_2-(t+1))$ and $(c_1, c_2-2(t+1))$ are codewords of $\mathcal{C}$, $\tau_1^{t+1}\tau_2^{t+1}(S)$ and $\tau_2^{2(t+1)}(S)$ are orthogonal to $\mathcal{I}_{\mathcal{C}, \mathcal{A}}$.
	\item[(2)] Note that
		\[
		(\tau_2^{2x_2+1}(s(S)))_{c_1, c_2}=(s(S))_{c_1, c_2-(2x_2+1)}=S_{c_1, n-c_2+2x_2}.
		\]
		Since $\mathbf{c}=\mathbf{x}+\alpha(t+1, t+1)+\beta(0, 2(t+1))$ for some $\alpha, \beta\in\mathbb{Z}_n$, it follows that
		\[
		(c_1, n-c_2+2x_2)=\mathbf{x}+\alpha(t+1, t+1)-(\alpha+\beta)(0, 2(t+1))
		\]
		is also a codeword of $\mathcal{C}$. So $(\tau_2^{2x_2+1}s(S))$ is orthogonal to $\mathcal{I}_{\mathcal{C}, \mathcal{A}}$.
	\item[(3)] Since a $180^\circ$ rotation reverses the positions of the two entries in the core of each translated anticode on the grid, it suffices to show that
		\[
		(\tau_1^{2x_1+2}\tau_2^{2x_2+1} r^2(S))_{c_1, c_2} = S_{c_1'+1, c_2'}
		\]
		and
		\[
		(\tau_1^{2x_1+2}\tau_2^{2x_2+1} r^2(S))_{c_1+1, c_2} = S_{c_1', c_2'}
		\]
		for some codeword $\mathbf{c}' = (c_1', c_2') \in \mathcal{C}$. Let $x'=2x_1+2$ and $x''=2x_2+1$. Then we compute:
		\[
			(\tau_1^{x'}\tau_2^{x''} (r^2(S)))_{c_1, c_2}=(r^2(S))_{c_1-x', c_2-x''}=S_{2x_1-c_1+1, 2x_2-c_2}.
		\]
		If we take $\mathbf{c}'=2\mathbf{x}-\mathbf{c}\in\mathcal{C}$, then we have
		\[
		(\tau_1^{2x_1+2}\tau_2^{2x_2+1} r^2(S))_{c_1, c_2} = S_{c_1'+1, c_2'}.
		\]
		Similarly,
		\[
			(\tau_1^{x'}\tau_2^{x''} r^2(S))_{c_1+1, c_2}=S_{2x_1-c_1, 2x_2-c_2}=S_{c_1', c_2'}.
		\]
		Therefore, $\tau_1^{2x_1+2-n} \tau_2^{2x_2+1-n} r^2(S)$ is orthogonal to $\mathcal{I}_{\mathcal{C}, \mathcal{A}}$.
	\end{itemize}
	Hence $g \cdot S \in \mathcal{S}$ for every generator $g \in \mathcal{G}_\mathcal{S}$, and the group $\mathcal{G}_\mathcal{S}$ acts on $\mathcal{S}$ with respect to the code $\mathcal{C}$.
\item[(ii)] Suppose that $\mathcal{C}$ is of Case II. As in the previous case, we verify that each generator of $\mathcal{G}_\mathcal{S}$ maps any $S \in \mathcal{S}$ to another element in $\mathcal{S}$. Take any $\mathbf{c} = (c_1, c_2) \in \mathcal{C}$ and $S \in \mathcal{S}$.

Note that
\[
(\tau_1^{a}\tau_2^{b}(S))_{c_1, c_2}=S_{c_1-a, c_2-b}.
\]
Since $\mathbf{c}-(a, b)\in\mathcal{C}$, it follows that $\tau_1^{a}\tau_2^{b}(S)$ is orthogonal to $\mathcal{I}_{\mathcal{C}, \mathcal{A}}$.

Note that the proof of (i)-(3) does not depend on whether $\mathcal{C}$ is of Case I or Case II. Therefore, as shown in (i), we have
\[
\mathcal{I}_\mathcal{C}\perp\tau_1^{2x_1+2} \tau_2^{2x_2+1} r^2(S).
\]
Thus $g\cdot S\in\mathcal{S}$ for all $g\in \mathcal{G}_\mathcal{S}$. This completes the proof.
\end{enumerate}
\end{proof}
Analogously to the perfect Sudoku grids over $\mathbb{Z}_5$, $\mathcal{G}_\mathcal{S}$ induces an equivalence relation on $\bar{\mathcal{S}}$ via its action. Figure~\ref{fig:8x8_operation_examples} illustrates the action of $\mathcal{G}_\mathcal{S}$ on a diameter perfect Sudoku grid $S$ with respect to Case I code $\mathcal{C}'$ over $\mathbb{Z}_8$.

\begin{figure}[ht]
\centering
\includegraphics[width=0.7\textwidth]{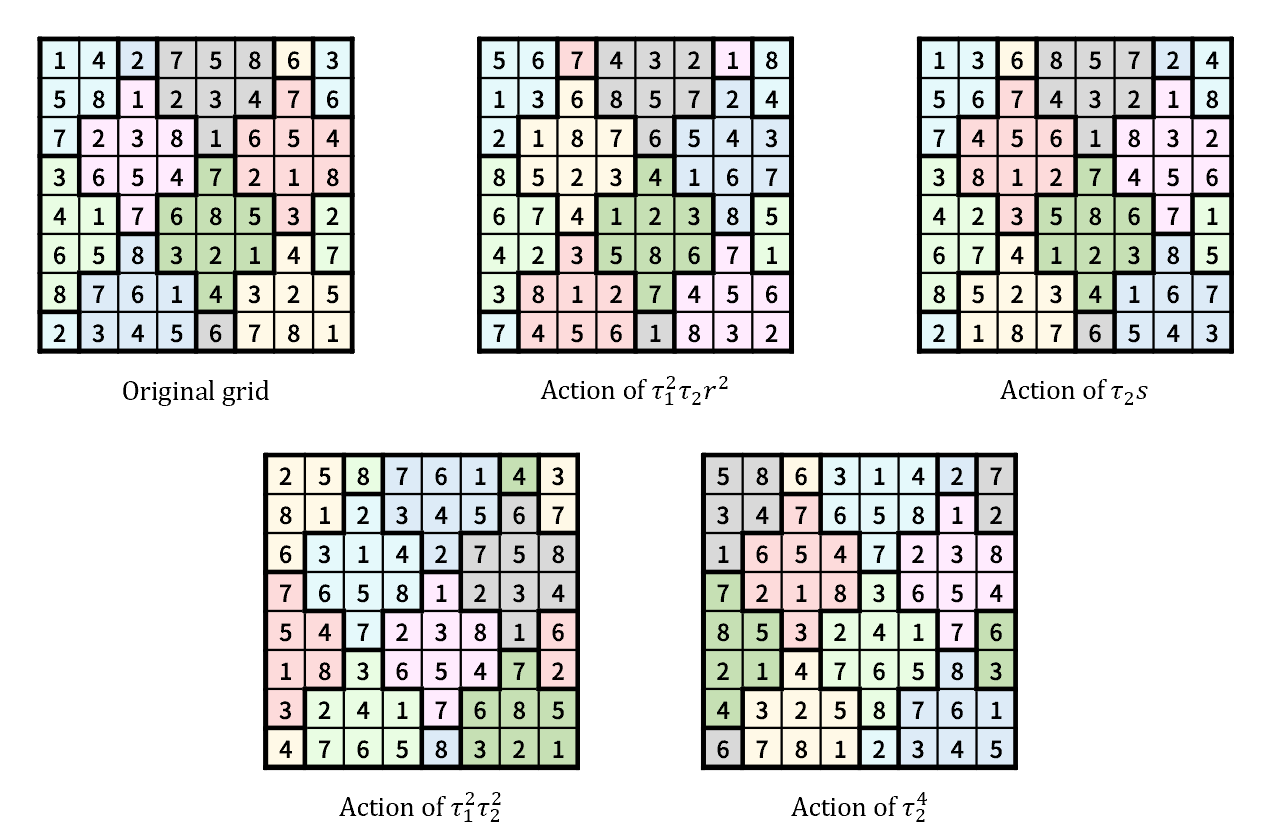}
\caption{Action of $\mathcal{G}_\mathcal{S}$ on a diameter perfect Sudoku grid with respect to $\mathcal{C}'$ over $\mathbb{Z}_8$}
\label{fig:8x8_operation_examples}
\end{figure}

\begin{theorem}
Let $\mathcal{C}_0$ be a linear $t$-error-correcting diameter perfect code with generator matrix $G_0$ and $\mathcal{C}=\mathbf{x}+\mathcal{C}_0\subseteq\mathbb{Z}_n^2$ for some $\mathbf{x}=(x_1, x_2)\in\mathbb{Z}_n^2$. Suppose $\mathcal{C}$ is either of Case I or Case II. Then the following holds:
\begin{enumerate}
\item[(i)] If $\mathcal{C}$ is of Case I and
	\[
	\mathcal{G}_\mathcal{S}=\langle \tau_1^{t+1}\tau_2^{t+1},\, \tau_2^{2(t+1)},\, \tau_2^{2x_2+1}s,\, \tau_1^{2x_1+2}\tau_2^{2x_2+1}r^2\rangle,
	\]
	then each equivalence class in $\bar{\mathcal{S}}$ under the action of $\mathcal{G}_\mathcal{S}$ has a size which is a divisor of $4n$.
\item[(ii)] If $\mathcal{C}$ is of Case II and
	\[
	\mathcal{G}_\mathcal{S}=\langle \tau_1^{a}\tau_2^{b},\, \tau_1^{2x_1+2}\tau_2^{2x_2+1}r^2\rangle
	\]
	where $G_0=[a~~b]$, then, each equivalence class in $\bar{\mathcal{S}}$ has a size which is a divisor of $2n$.
\end{enumerate}
\end{theorem}
\begin{proof}
We compute the order of $\mathcal{G}_\mathcal{S}$ in each case. In case (i), we will show that $|\mathcal{G}_\mathcal{S}| = 4n$, and in case (ii), that $|\mathcal{G}_\mathcal{S}| = 2n$. By the orbit-stabilizer theorem, each equivalence class in $\bar{\mathcal{S}}$ has size dividing $|\mathcal{G}_\mathcal{S}|$.
\begin{enumerate}
\item[(i)] Note that
\[
|\tau_1|=|\tau_2|=n, |r|=4\quad\mbox{and}\quad |s|=2.
\]
Since $\tau_1\tau_2=\tau_2\tau_1$, the order of $\tau_1^{t+1}\tau_2^{t+1}$ is given by
\[
|\tau_1^{t+1}\tau_2^{t+1}|=|(\tau_1\tau_2)^{t+1}|=\frac{n}{t+1}=2(t+1).
\]
Similarly, the order of $\tau_2^{2(t+1)}$ is
\[
|\tau_2^{2(t+1)}|=\frac{n}{2(t+1)}=t+1.
\]
Next, we have
\[
((\tau_2^{2x_2+1}s)^2(S))_{i, j}=(\tau_2^{2x_2+1}s(S))_{i, n-j+2x_2}=S_{i, j},
\]
which gives $|\tau_2^{2x_2+1}s|=2$. Finally, consider the element $\tau_1^{2x_1 + 2} \tau_2^{2x_2 + 1} r^2$. Let $x'=2x_1 + 2$ and $x''=2x_2 + 1$. Then we compute:
	\[
	((\tau_1^{x'}\tau_2^{x''}r^2)^2(S))_{i, j}=(\tau_1^{x'}\tau_2^{x''}r^2(S))_{2x_1-i+1, 2x_2-j}=S_{i ,j}.
	\]
	Thus, $|\tau_1^{2x_1 + 2} \tau_2^{2x_2 + 1} r^2|=2$. It follows that	$|\mathcal{G}_\mathcal{S}|=2(t+1)^2\cdot 2\cdot 2=4n.$
\item[(ii)] Note that $(a, b)$ is either $(1,\, 2t+1)$ or $(2t+1,\, 1)$. Thus $|\tau_1^{a} \tau_2^{b}| = n$. As shown in case (i), $|\tau_1^{2x_1 + 2} \tau_2^{2x_2 + 1} r^2| = 2$. Therefore, $|\mathcal{G}_\mathcal{S}| = 2n$.
\end{enumerate}
This completes the proof. \end{proof}

\section{Solutions of some specific Sudoku}
In this section, we compute the number of (diameter) perfect Sudoku grids with respect to (diameter) perfect codes over certain small rings of integer modulo $n$. Furthermore, we classify these Sudoku grids into equivalence classes under the appropriate group actions.

For codes of length $2$, the equivalence under coordinate permutation corresponds precisely to rotations and reflections. Thus, $\mathcal{C}_1$ and $\mathcal{C}_2$ are related by a rigid motion, which induces a bijection between their corresponding sets of Sudoku grids. This proves the following.

\begin{theorem}\label{eqcode}
Let $\mathcal{C}_1$ and $\mathcal{C}_2$ be two (diameter) perfect codes in $\mathbb{Z}_n^2$ such that $\mathcal{C}_1$ is a translate of a code equivalent to $\mathcal{C}_2$. Then the number of (diameter) perfect Sudoku grids that can be constructed from $\mathcal{C}_1$ is equal to the number of (diameter) perfect Sudoku grids that can be constructed from $\mathcal{C}_2$.
\end{theorem}



\begin{definition}
Let $S$ be a Sudoku grid with respect to some palette grid $\mathcal{I}$. A \textit{minimal Sudoku} $m$ for $S$ is a Sudoku game that satisfies the following two conditions:
\begin{enumerate}
\item[(i)] $S$ is the unique solution of $m$.
\item[(ii)] If any single hint (filled entry) is removed, the resulting game no longer has a unique solution.
\end{enumerate}
\end{definition}

Since the action of $\mathcal{G}_\mathcal{S}$ preserves the Sudoku property and the inclusion of hints, it follows that the minimal Sudoku property is also preserved. We present the following theorem, whose detailed proof is omitted.

\begin{theorem}\label{eqSudoku}
Let $S_1, S_2 \in \mathcal{S}$ be two (diameter) perfect Sudoku grids with respect to some (diameter) perfect code $\mathcal{C}$, which are equivalent under the action of $\mathcal{G}_\mathcal{S}$, i.e., $S_2=\phi(S_1)$ for some $\phi\in\mathcal{G}_\mathcal{S}$. Then $\phi$ is a bijection between the sets of minimal Sudoku for $S_1$ and $S_2$.
\end{theorem}
	
	


\subsection{Perfect Sudoku from perfect codes over \texorpdfstring{$\mathbb{Z}_5$}{Z5}}

In this subsection, we consider perfect Sudoku grids with respect to a $1$-error-correcting perfect code over $\mathbb{Z}_5$. The underlying perfect code is $\mathcal{C}=(2, 2) + \mathcal{C}_0$ where $\mathcal{C}_0\subseteq\mathbb{Z}_5^2$ is a linear code with generator matrix $G=[3~~1]$. However, perfect Sudoku grids with respect to $\mathcal{C}$ represent all perfect Sudoku grids arising from perfect codes over $\mathbb{Z}_5$ by Theorem~\ref{eqcode}.

We found 2,040 perfect Sudoku grids satisfying the constraints of perfect Sudoku grids over $\mathbb{Z}_5$. Taking relabeling into account, these 2,040 grids are partitioned into $\frac{2040}{5!} = 17$ equivalence classes. From Theorem~\ref{pSudoku}, there is a well-defined group action of $\mathcal{G}_\mathcal{S}$. With $\mathbf{x}=(2, 2)$ and $G=[3 \;\; 1]$, we obtain $\mathcal{G}_\mathcal{S}=\langle r, \tau_1^3\tau_2\rangle$. We further analyze the relabeling equivalence classes $\bar{\mathcal{S}}$ under the actions of rotation $\langle r \rangle$, translation $\langle \tau_1^3\tau_2 \rangle$, and the full group $\langle r, \tau_1^3\tau_2\rangle$.

\begin{table}[ht]
\centering
\renewcommand{\arraystretch}{1.4}
\begin{tabular}{|c|c|c|c|}
\hline
\textbf{Class size} & $\langle r \rangle$ & $\langle \tau_1^3\tau_2\rangle$ & $\langle r, \tau_1^3\tau_2\rangle$ \\
\hline
10 &  &  & 1 \\
5  &  & 3  & 1 \\
4  & 3  &  &  \\
2  & 1  &  &  \\
1  & 3  & 2  & 2 \\
\hline
\textbf{Total} & \textbf{7} & \textbf{5} & \textbf{4} \\
\hline
\end{tabular}
\caption{Equivalence class size distribution of \\perfect Sudoku grids with respect to $\mathcal{C}$ over $\mathbb{Z}_5$}
\label{tab:5x5-equiv-summary}
\end{table}

Table~\ref{tab:5x5-equiv-summary} presents the classification of perfect Sudoku grids derived from perfect codes over $\mathbb{Z}_5$. As expected, the rotation classes have sizes $1$, $2$, or $4$, translation classes have sizes $1$ or $5$, and classes under the combined action of rotation and translation have sizes that divide $10$.

The equivalence classes under the action $\langle r, \tau_1^3\tau_2\rangle$ are illustrated in Figure~\ref{fig:combined_classes_visualization_2}. Since $|\langle r, \tau_1^3\tau_2\rangle|=10$, each equivalence class has a size that divides $10$. Class $1$ consists of perfect Sudoku grids for which every application of $\langle r, \tau_1^3\tau_2\rangle$ produces a distinct grid, whereas class $2$ consists of perfect Sudoku grids such that the application of $\tau_1^3\tau_2$ gives a distinct grid. The grids in classes $3$ and $4$ are fixed points under the action of $\langle r, \tau_1^3\tau_2\rangle$.

\begin{figure}[ht]
\centering
\includegraphics[width=0.7\textwidth]{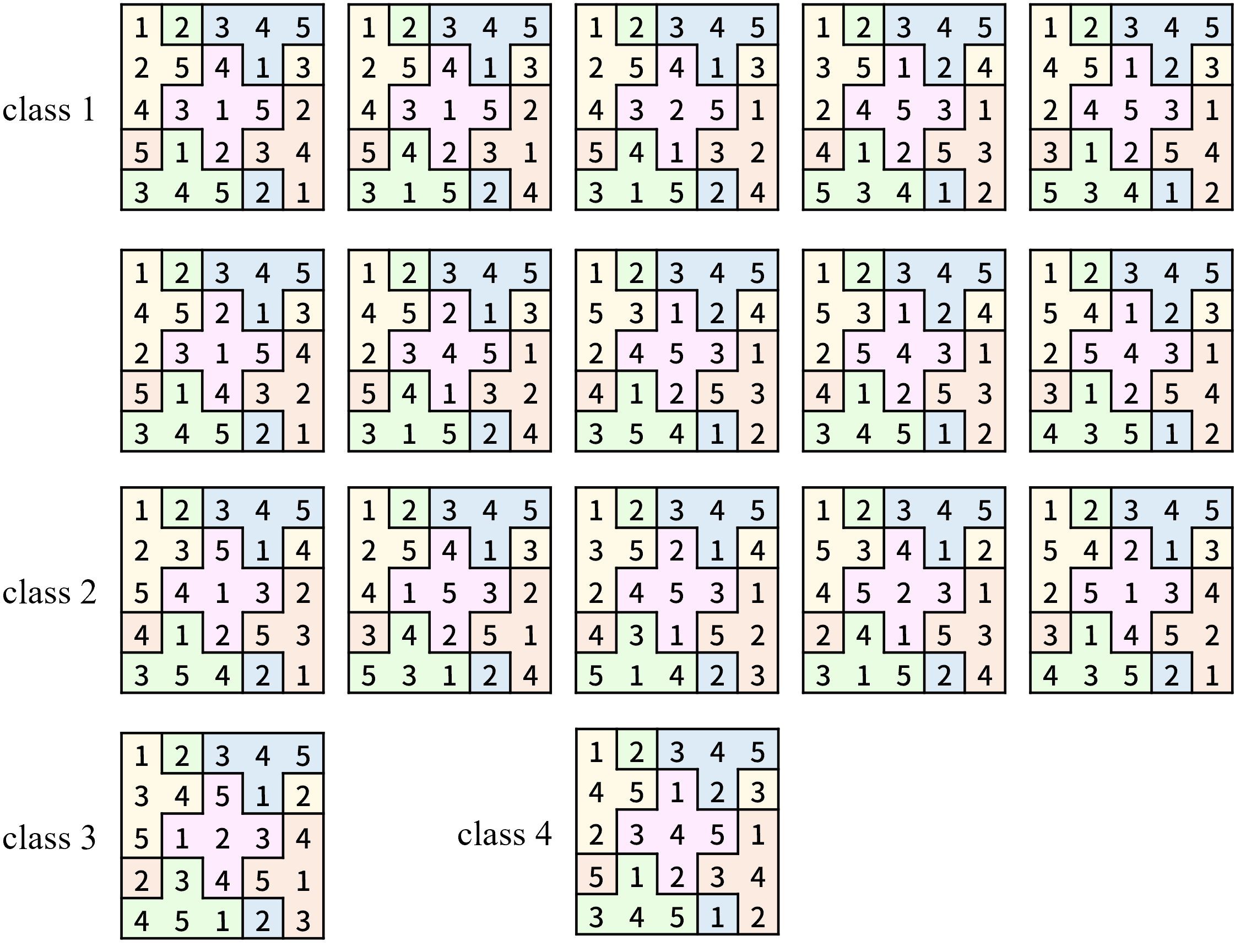}
\caption{All equivalence classes under the action of $\langle r, \tau_1^3\tau_2\rangle$}
\label{fig:combined_classes_visualization_2}
\end{figure}

We observe that for each grid in class 3 in Figure~\ref{fig:combined_classes_visualization_2}, every ball of radius 1 in the grid contains each of the numbers 1 through 5 exactly once, regardless of the center points of the balls. This implies that the grid is not only a Sudoku grid with respect to $\mathcal{C}$, but also a Sudoku grid with respect to any perfect code of length $2$ and minimum distance $3$ by regarding the codewords as the center points. This motivates the following definition.

An array $S$ is called a \textit{special Sudoku grid over $\mathbb{Z}_5$} if it is a Sudoku grid with respect to any $(2, 5, 3)$ perfect code over $\mathbb{Z}_5$.

We can check that Sudoku grids in class 3 and class 4 are special Sudoku grids. Conversely, we show that any special Sudoku grid over $\mathbb{Z}_5$ must belong to either class 3 or class 4 as follows. Suppose that $S$ is a special Sudoku grid over $\mathbb{Z}_5$, and for each $1 \le k \le 5$, define the set of positions for value $k$ by
\[S_k=\{(i, j)\in\mathbb{Z}_5^2~|~S_{i, j}=k\}.\]
Since every ball of radius 1 contains exactly one position from $S_k$, the minimum Lee distance between any two distinct points in $S_k$ must be at least 3. This distance property identifies $S_k$ as a $(2, 5, 3)$ perfect code over $\mathbb{Z}_5$ for each $k$. Then $S_k=\mathbf{x}+\mathcal{C'}$ for some $\mathbf{x} \in \mathbb{Z}_5^2$ and a linear $(2, 5, 3)$ perfect code $\mathcal{C'}$ by Theorem~\ref{pcode}.

Let $\mathcal{C}_0$ and $\mathcal{C}_1$ be two distinct linear $(2, 5, 3)$ perfect codes over $\mathbb{Z}_5$. Since these codes are one-dimensional subspaces of $\mathbb{Z}_5^2$, they span $\mathbb{Z}_5^2$. Then, for any translates $\mathbf{u}+\mathcal{C}_0$ of $\mathcal{C}_0$ and $\mathbf{v}+\mathcal{C}_1$ of $\mathcal{C}_1$, we have
\[
\mathbf{u}-\mathbf{v}=\mathbf{c}_1-\mathbf{c}_0
\]
for some $\mathbf{c}_0\in\mathcal{C}_0$ and $\mathbf{c}_1\in\mathcal{C}_1$. This implies that $\mathbf{u}+\mathbf{c}_0=\mathbf{v}+\mathbf{c}_1$, that is, $\mathbf{u}+\mathcal{C}_0$ and $\mathbf{v}+\mathcal{C}_1$ are not disjoint for any $\mathbf{u}, \mathbf{v}\in\mathbb{Z}_5^2$. Since the sets $\{S_k\}_{k=1}^5$ are pairwise disjoint by definition, each $S_k$ must be a translate of the same linear $1$-error-correcting perfect code $\mathcal{C}_0\subseteq\mathbb{Z}_5^2$, namely,
\[
S_k = \mathbf{u}_k + \mathcal{C}_0
\]
for some $\mathbf{u}_k \in \mathbb{Z}_5^2$. Since $S_1, \dots, S_5$ are translates of one another, any translation of the grid $S$ is equivalent to $S$ up to relabeling.

We next examine the effect of the rotation on a special Sudoku grid. Since the rotation of $S$ preserves the distances between elements of each $S_k$, it sends a perfect code into a perfect code. Also, if $(a, b)$ is a generator of $\mathcal{C}_0$, then the rotated perfect code is $\mathbf{x}+\langle(-b, a)\rangle=\mathbf{x}+\langle(a, b)\rangle=\mathbf{x}+\mathcal{C}_0$ for some $\mathbf{x}\in\mathbb{Z}_n^2$. This implies that the rotation of $S$ is equivalent to $S$ up to relabeling. Therefore, special Sudoku grids over $\mathbb{Z}_5$ must be either of class 3 or class 4.

Note that the reflection of a special Sudoku grid is also a special Sudoku grid. Since each class has $120$ grids, and class 3 is reflection of class 4, we have the following theorem.
\begin{theorem}
    There are $240$ special Sudoku grids over $\mathbb{Z}_5$. These grids are all equivalent to the representative of class 3 up to relabeling and reflection.
\end{theorem}

We enumerated all minimal Sudoku for the $5 \times 5$ perfect Sudoku grids, where the number $k$ of given hints ranges from $4$ to $7$. Examples of the cases $k=4$ and $k=7$ are shown in Figure~\ref{fig:5x5-minimal-example}.
\begin{figure}[ht]
\centering
\includegraphics[width=0.5\textwidth]{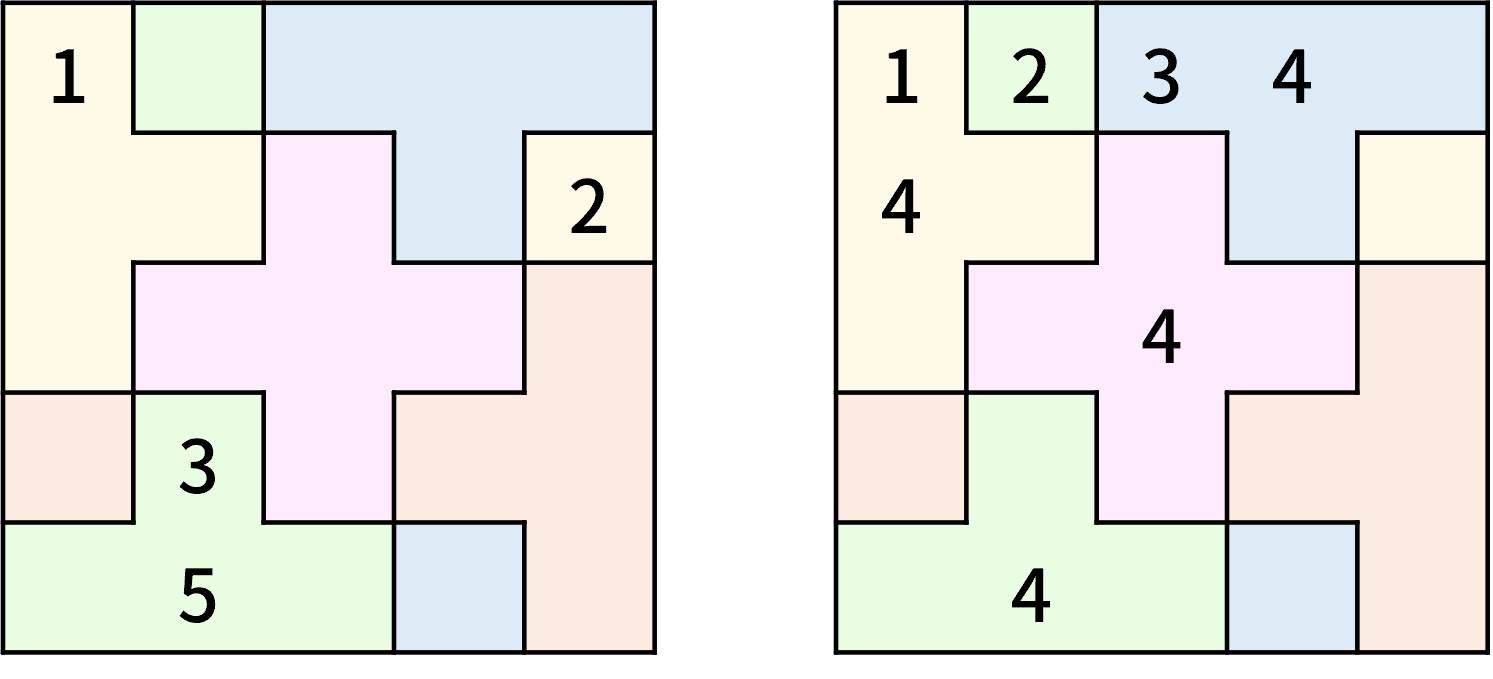}
\caption{Examples of minimal Sudoku from a perfect code over $\mathbb{Z}_5$ with $k = 4$ and $k=7$}
\label{fig:5x5-minimal-example}
\end{figure}

By Theorem~\ref{eqSudoku}, instead of exhaustively checking all possible hint placements on the $2{,}040$ valid perfect Sudoku grids, it is sufficient to search for the four representative perfect Sudoku grids of the equivalence classes derived from $\langle r, \tau_1^3\tau_2\rangle$. The resulting count of minimal Sudoku for each number of given hints is summarized in Table~\ref{tab:minimal-Sudoku-counts}.

\begin{table}[ht]
\renewcommand{\arraystretch}{1.4}
\centering
\begin{tabular}{|c|c|c|}
\hline
\textbf{$k$} & \textbf{Number of minimal Sudoku} & \textbf{Up to equivalence} \\
\hline
4 & 154,200     & 507    \\
5 & 5,721,600 & 14,860 \\
6 & 8,908,800 & 19,096 \\
7 & 1,113,600 & 1,296  \\
\hline
\textbf{Total} & \textbf{15,898,200} & \textbf{35,759} \\
\hline
\end{tabular}
\caption{Number of minimal Sudoku from perfect codes over $\mathbb{Z}_5$ by hint count $k$}
\label{tab:minimal-Sudoku-counts}
\end{table}

The counts in Table~\ref{tab:minimal-Sudoku-counts} were computed using Algorithm~\ref{alg:minimal}, which systematically evaluates all \(k\)-hint games generated from each representative class. For each of the four classes, a representative grid was selected, and all \(\binom{25}{k}\) games with \(k\) hints were checked for uniqueness and minimality. The resulting class-wise counts were then weighted by their respective equivalence class sizes: 120 for classes 1 and 2, 600 for class 3, and 1,200 for class 4.

\SetAlgoBlockMarkers{begin}{end}
\SetKwIF{If}{ElseIf}{Else}{if}{then}{else if}{else}{end if}
\SetKwFor{For}{for}{do}{end for}
\SetKwFor{While}{while}{do}{end while}

\begin{algorithm}
\caption{Number of minimal Sudoku by hint count \(k\)}\label{alg:minimal}
\KwData{\(k\)}
\KwResult{\(T\)}

Initialize total count \(T = 0\)

\For{each class}{

  Initialize \(T_c = 0\);

  Select a representative solution \(P_c\);

  Generate all \(P_{c,k}\) obtained from \(P_c\) by choosing \(k\) hints;

  \For{each game \(P_{c,k}\)}{
    Check if \(P_{c,k}\) has a unique solution;

    Check if all games obtained by removing one hint from \(P_{c,k}\) do not have a unique solution;

    If both, count \(P_{c,k}\) as a minimal game and increment \(T_c\);
  }

  Multiply \(T_c\) by the class size of \(c\);

  Add to total count: \(T \leftarrow T + T_c\);
}
\end{algorithm}

\subsection{Diameter perfect Sudoku from diameter perfect codes over \texorpdfstring{$\mathbb{Z}_8$}{Z8}}

We now consider diameter perfect Sudoku grids from $3$-diameter perfect codes $\mathcal{C}$ and $\mathcal{C}'$ whose generator matrices are given as follows:
\[
G=\begin{bmatrix}
2 & 2\\0 & 4	
\end{bmatrix}
\quad\mbox{and}\quad
G'=[1, 3].
\]
For each code $\mathcal{C}$ and $\mathcal{C}'$, we define the associated anticodes $\mathcal{A}$ and $\mathcal{A}'$, respectively. The core of each anticode corresponding to a codeword $\mathbf{c} \in \mathcal{C}$ (respectively, $\mathcal{C}'$) is defined by $\{\mathbf{c}, (1, 0)+\mathbf{c}\}$.

The corresponding palette grids $I_{\mathcal{C},\, \mathcal{A}}$ and $I_{\mathcal{C}',\, \mathcal{A}'}$ derived from the diameter perfect codes $\mathcal{C}$ and $\mathcal{C}'$ are depicted in Figure~\ref{fig:3-diameter perfect code}.

\begin{figure}[ht]
\centering
\includegraphics[width=0.7\textwidth]{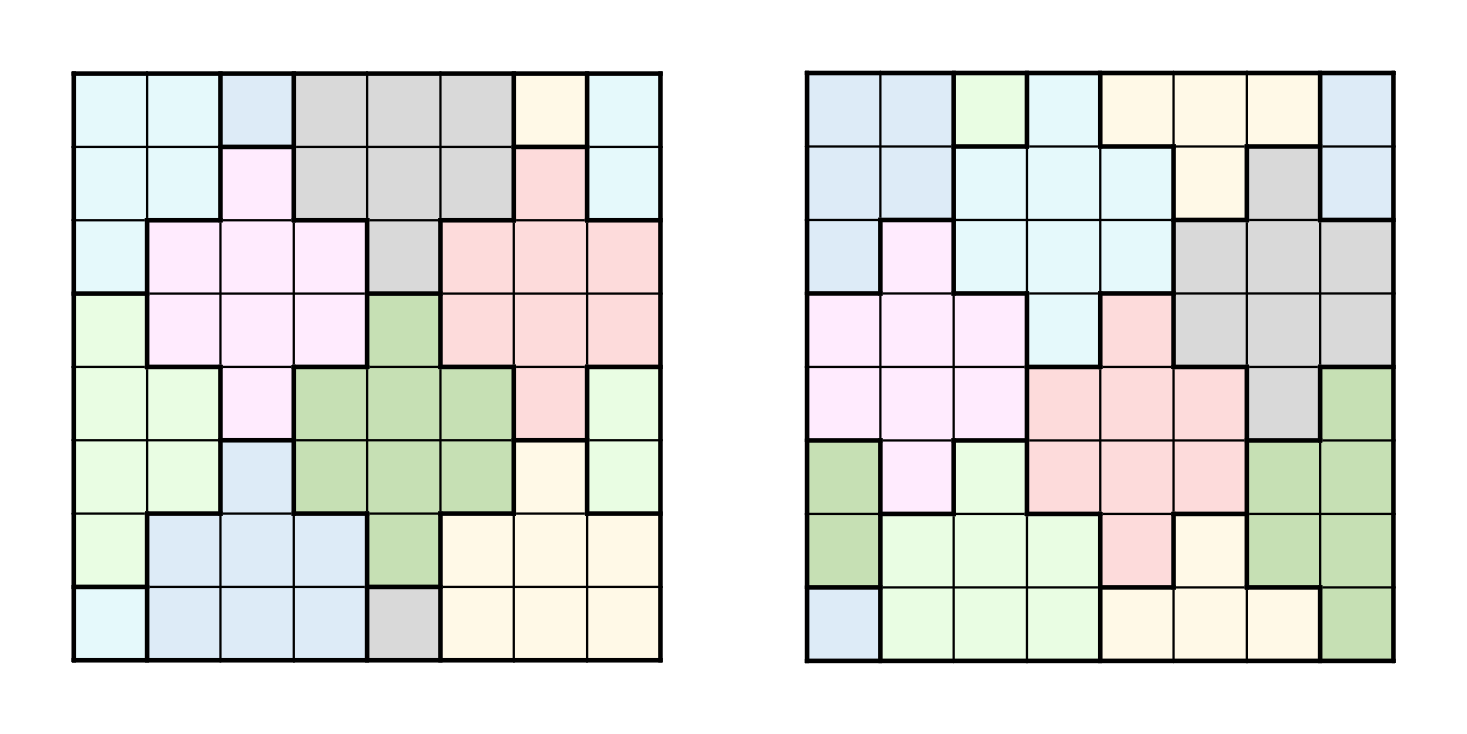}
\caption{Palette grids $I_C$, $I_{C'}$ of the $3$-diameter perfect codes}
\label{fig:3-diameter perfect code}
\end{figure}

We begin with the analysis of diameter perfect Sudoku grids derived from the diameter perfect code $\mathcal{C}$ and its corresponding anticode $\mathcal{A}$. Using a backtracking algorithm, we found a total of 6,940,096 solutions up to relabeling. Accounting for all relabelings (a factor of $8!$), this corresponds to $6,\!940,\!096 \times 8! \approx 2.80 \times 10^{11}$ valid grids.

Next, we classify the diameter perfect  Sudoku grids according to the equivalence relations induced by $\langle \tau_1^2\tau_2^2, \tau_2^4 \rangle$, $\langle \tau_2s \rangle$, $\langle \tau_1^2\tau_2r^2 \rangle$, and $\langle \tau_1^2\tau_2^2, \tau_2^4,\tau_2s, \tau_1^2\tau_2r^2\rangle$.

Table~\ref{tab:8x8-CA-symmetry-comparison} presents the distribution of equivalence classes under the actions.

\begin{table}[ht]
\renewcommand{\arraystretch}{1.4}
\centering
\begin{tabular}{|c|c|c|c|c|}
\hline
\rule{0pt}{15pt}\textbf{Class size} & $\langle \tau_1^2\tau_2r^2 \rangle$ & $\langle \tau_2s \rangle$ & $\langle \tau_1^2\tau_2^2, \tau_2^4 \rangle$ &  $\langle \tau_1^2\tau_2^2, \tau_2^4, \tau_2s, \tau_1^2\tau_2r^2\rangle$ \\
\hline
32 &               &               &         & 202,658 \\
16 &               &               &         & 26,927 \\
8  &               &               & 844,905 & 2,906 \\
4  &               &               & 44,411  & 236 \\
2  & 3,467,680     & 3,470,048     & 1,554   & 8 \\
1  & 4,736         &               & 104     & \\
\hline
\textbf{\textbf{\begin{tabular}{c}Distinct\\[-4pt]classes\end{tabular}}} & \textbf{3,472,416} & \textbf{3,470,048} & \textbf{890,974} & \textbf{232,735} \\
\hline
\end{tabular}
\caption{Equivalence class size of diameter perfect Sudoku grids with respect to $\mathcal{C}$ over $\mathbb{Z}_8$}
\label{tab:8x8-CA-symmetry-comparison}
\end{table}

Next, we turn our attention to the diameter perfect Sudoku grids constructed from the diameter perfect code $\mathcal{C}'$ and its anticode $\mathcal{A}'$. Using the same backtracking algorithm applied in the case of $\mathcal{C}$ and $\mathcal{A}$, we found a total of $4,\!839,\!127$ distinct solutions up to relabeling. Accounting for the $8!$ possible symbol relabelings, this corresponds to $4,\!839,\!127 \times 8! \approx 1.95 \times 10^{11}$ valid diameter perfect Sudoku grids.

We further classified these diameter perfect  Sudoku grids under the actions of $\langle \tau_1\tau_2^3 \rangle$, $\langle \tau_1^2 \tau_2 r^2 \rangle$, $\langle \tau_1\tau_2^3, \tau_1^2 \tau_2 r^2 \rangle$. The number of equivalence classes, grouped by their sizes, is shown in Table~\ref{tab:8x8-CpAp-symmetry-comparison}.

\begin{table}[ht]
\centering
\renewcommand{\arraystretch}{1.4}
\begin{tabular}{|c|c|c|c|}
\hline
\textbf{Class size} & $\langle \tau_1^2 \tau_2 r^2 \rangle$ & $\langle \tau_1\tau_2^3 \rangle$ & $\langle \tau_1\tau_2^3, \tau_1^2 \tau_2 r^2 \rangle$ \\
\hline
16 &        &        & 300,965 \\
8  &        & 603,554  & 2,886 \\
4  &        & 2,601    & 139 \\
2  & 2,418,613 & 132     & 19 \\
1  & 1,901    & 27       & 5 \\
\hline
\textbf{Distinct classes} & \textbf{2,420,514} & \textbf{606,314} & \textbf{304,014} \\
\hline
\end{tabular}
\caption{Equivalence class size of diameter perfect Sudoku grids with respect to $\mathcal{C}'$ over $\mathbb{Z}_8$}
\label{tab:8x8-CpAp-symmetry-comparison}
\end{table}

Analogously to the $\mathbb{Z}_5$ case, we define a special Sudoku grid with respect to $\mathbb{Z}_8$ as a grid that is a diameter Sudoku grid with respect to any $(2, 8, 4)$-diameter perfect code over $\mathbb{Z}_8$.

Through a computational search on the $304,014$ classes of diameter perfect Sudoku grids (defined with respect to $\mathcal{C}'$ for representative convenience), we verified that there are three classes of special Sudoku grids. The two classes of size 1 are reflections of each other, while the class of size 2 is invariant under reflection. From this, we obtain the following.
\begin{theorem}
    There are $4\times 8!$ special Sudoku grids over $\mathbb{Z}_8$. Up to relabeling and the action of $\mathcal{G}'=\langle s, r^2\rangle$, there exist exactly two such inequivalent special Sudoku grids. Two special Sudoku grids $S$ and $S'$ over $\mathbb{Z}_8$ corresponding to class size 1 and 2, respectively, are given as follows:
    \[
S=\begin{bmatrix}
1 & 2 & 3 & 4 & 5 & 6 & 7 & 8\\
4 & 5 & 6 & 7 & 8 & 1 & 2 & 3\\
7 & 8 & 1 & 2 & 3 & 4 & 5 & 6\\
2 & 3 & 4 & 5 & 6 & 7 & 8 & 1\\
5 & 6 & 7 & 8 & 1 & 2 & 3 & 4\\
8 & 1 & 2 & 3 & 4 & 5 & 6 & 7\\
3 & 4 & 5 & 6 & 7 & 8 & 1 & 2\\
6 & 7 & 8 & 1 & 2 & 3 & 4 & 5
\end{bmatrix},
\]
and
\[
S'=\begin{bmatrix}
1 & 2 & 3 & 4 & 5 & 6 & 7 & 8\\
4 & 7 & 6 & 1 & 8 & 3 & 2 & 5\\
3 & 8 & 5 & 2 & 7 & 4 & 1 & 6\\
2 & 1 & 4 & 3 & 6 & 5 & 8 & 7\\
5 & 6 & 7 & 8 & 1 & 2 & 3 & 4\\
8 & 3 & 2 & 5 & 4 & 7 & 6 & 1\\
7 & 4 & 1 & 6 & 3 & 8 & 5 & 2\\
6 & 5 & 8 & 7 & 2 & 1 & 4 & 3
\end{bmatrix}.
\]
\end{theorem}

Unlike the perfect Sudoku from perfect codes over $\mathbb{Z}_5$, enumerating all minimal Sudoku from
diameter perfect codes over $\mathbb{Z}_8$ is computationally infeasible. Instead, we identified a variety of minimal Sudoku with the number $k$ of given hints ranging from $10$ to $18$. Examples of the cases $k=10$ and $k=18$ are shown in Figure~\ref{fig:8x8-minimal-example}.

\begin{figure}[ht]
\centering
\includegraphics[width=0.7\textwidth]{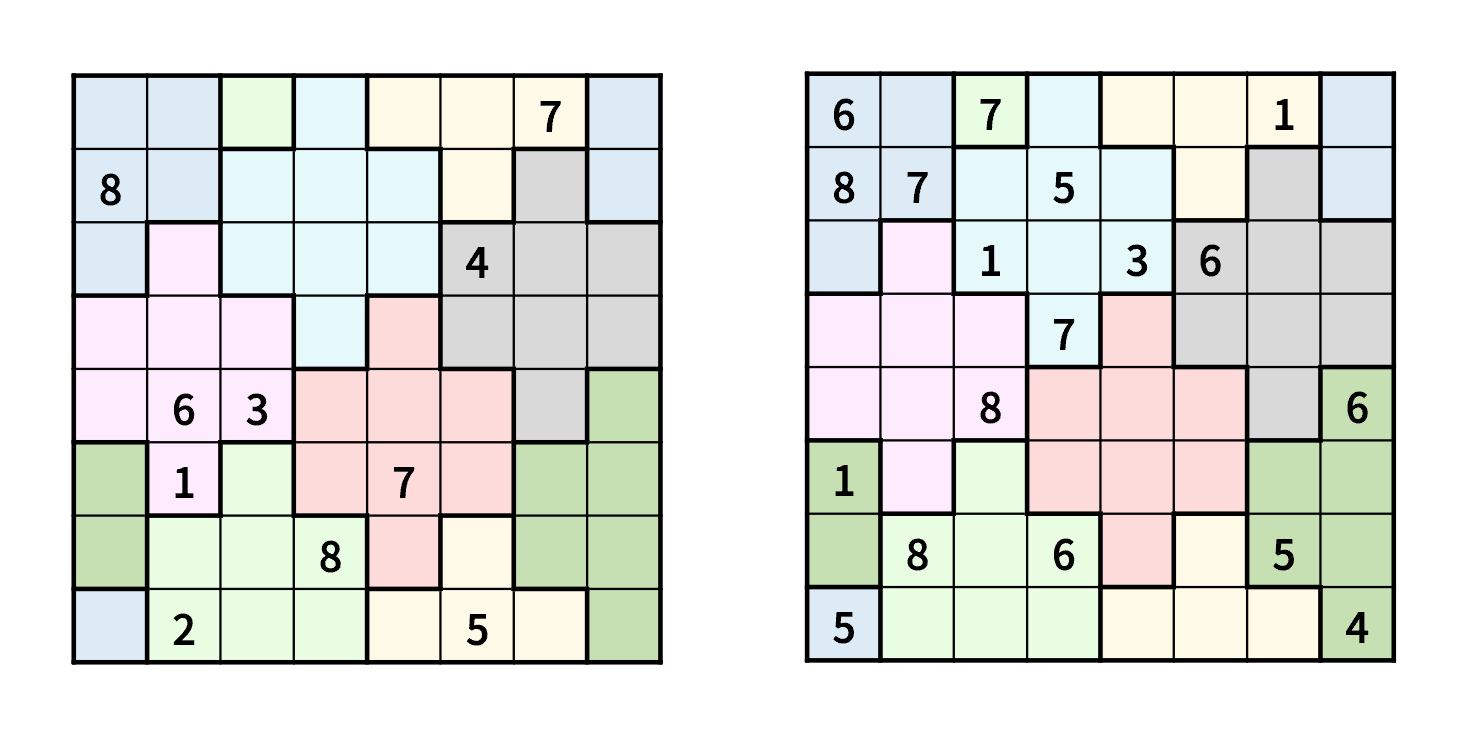}
\caption{Examples of minimal Sudoku from diameter perfect codes over $\mathbb{Z}_8$ with $k = 10$ and $k=18$}
\label{fig:8x8-minimal-example}
\end{figure}

\section{Practical game aspects of the \texorpdfstring{$5 \times 5$}{5x5} perfect Sudoku}
To assess our newly constructed $5 \times 5$ perfect Sudoku as a practical game, we analyzed its solving complexity using a naive human-like solver that mimics logical deduction processes typically employed by general human solvers without special strategies. This solver allowed us to measure the number of trials required to complete minimal Sudoku configurations, providing a metric for difficulty and challenge. The goal was to determine whether these games possess meaningful levels of challenge and diversity in difficulty.

Specifically, we applied this algorithm to the entire dataset of $5 \times 5$ nonequivalent minimal Sudoku obtained from our constructions given in Table~\ref{tab:minimal-Sudoku-counts} and measured the number of digit insertions required to reach a complete solution for each instance.

The naive human-like solver employed in our analysis operates in two alternating phases: the deterministic phase and the trial phase.

The deterministic phase is based on the two most fundamental human-like solving techniques, as described in~\cite{CFD18}:
\begin{enumerate}
  \item \textbf{Naked Single:} The algorithm scans each blank cell. If a cell has only one possible candidate value after checking the
constraints of its row, column, and subgrid, then that value is assigned.
  \item \textbf{Hidden Single:} For each digit $d$, the algorithm identifies all possible cells where $d$ could be placed. If, within any given
row, column, or subgrid, there is only one such cell, then $d$ is assigned to it.
\end{enumerate}
These two procedures are applied iteratively until the grid reaches a fixed point—a state where no further assignments can be made using the deterministic rules.

The trial phase is initiated only when the deterministic approach proves insufficient to solve the game. This phase begins when the deterministic process reaches a fixed point while the grid remains incomplete. This phase emulates a human's trial-and-error process, guided by a common heuristic to make an educated guess. The procedure is as follows:
\begin{enumerate}
  \item \textbf{Candidate Analysis:} The algorithm first determines the set of valid candidates for each blank cell based on the current grid state.
  \item \textbf{Heuristic-Based Cell Selection:} It then identifies the cell (or cells) with the fewest possible candidates—the most constrained cell. From this set of most-constrained cells, one is selected uniformly at random.
  \item \textbf{Tentative Assignment:} Finally, the algorithm randomly chooses one of the candidate digits for the selected cell and assigns it.
\end{enumerate}
Following this assignment, the deterministic phase is reapplied iteratively until it once again reaches a fixed point. This alternation between a tentative assignment and deterministic deduction continues until the game is solved.

If a guess leads to a contradiction (e.g., a cell has no valid candidates), the assignment is considered a failure. The algorithm then backtracks, reverting the grid to its state before the incorrect guess and attempting a different candidate. The number of such trial-and-error attempts serves as a practical measure of a game's difficulty.

For each minimal Sudoku, we applied our naive human-like solver 100 times and recorded the number of digit insertions required to reach a complete solution. We posted the solver in Python at https://cicagolab.pythonanywhere.com.
The average number of insertions per game was then adjusted by subtracting the number of initially blank cells. We refer to this number as the \textit{difficulty score}. Under this definition, games solvable purely by the deterministic phase have a difficulty score of zero. The resulting difficulty scores range from 0 to 38.22. For example, if the score is 10, then ten more insertions than initially blank cells are made. Based on this score, games are classified into three levels: easy if the score is 0, medium if it is greater than 0 and at most 10, and hard if it exceeds 10.
The results are presented in Table~\ref{tab:difficulty_distribution} and Figure~\ref{score_scatter}.

\begin{table}[ht]
\renewcommand{\arraystretch}{1.4}
\centering
\begin{tabular}{|c|c|c|c|c|}
\hline
$k$ & \textbf{Easy} & \textbf{Medium} & \textbf{Hard} & \textbf{Total} \\[0.15em]
\hline
4 & 219 (43.19\%) & 106 (20.9\%) & 182 (35.9\%) & 507 \\[0.5em]

5 & 8,868 (59.67\%) & 4,274 (28.76\%) & 1,718 (11.56\%) & 14,860 \\[0.5em]

6 & 11,270 (59.01\%) & 7,175 (37.57\%) & 651 (3.41\%) & 19,096 \\[0.5em]

7 & 1,020 (78.7\%) & 274 (21.14\%) & 2 (0.15\%) & 1,296\\
\hline
\end{tabular}
\caption{Distribution of game difficulties by level}
\label{tab:difficulty_distribution}
\end{table}

\begin{figure}[ht]
\centering
\includegraphics[width=0.7\textwidth]{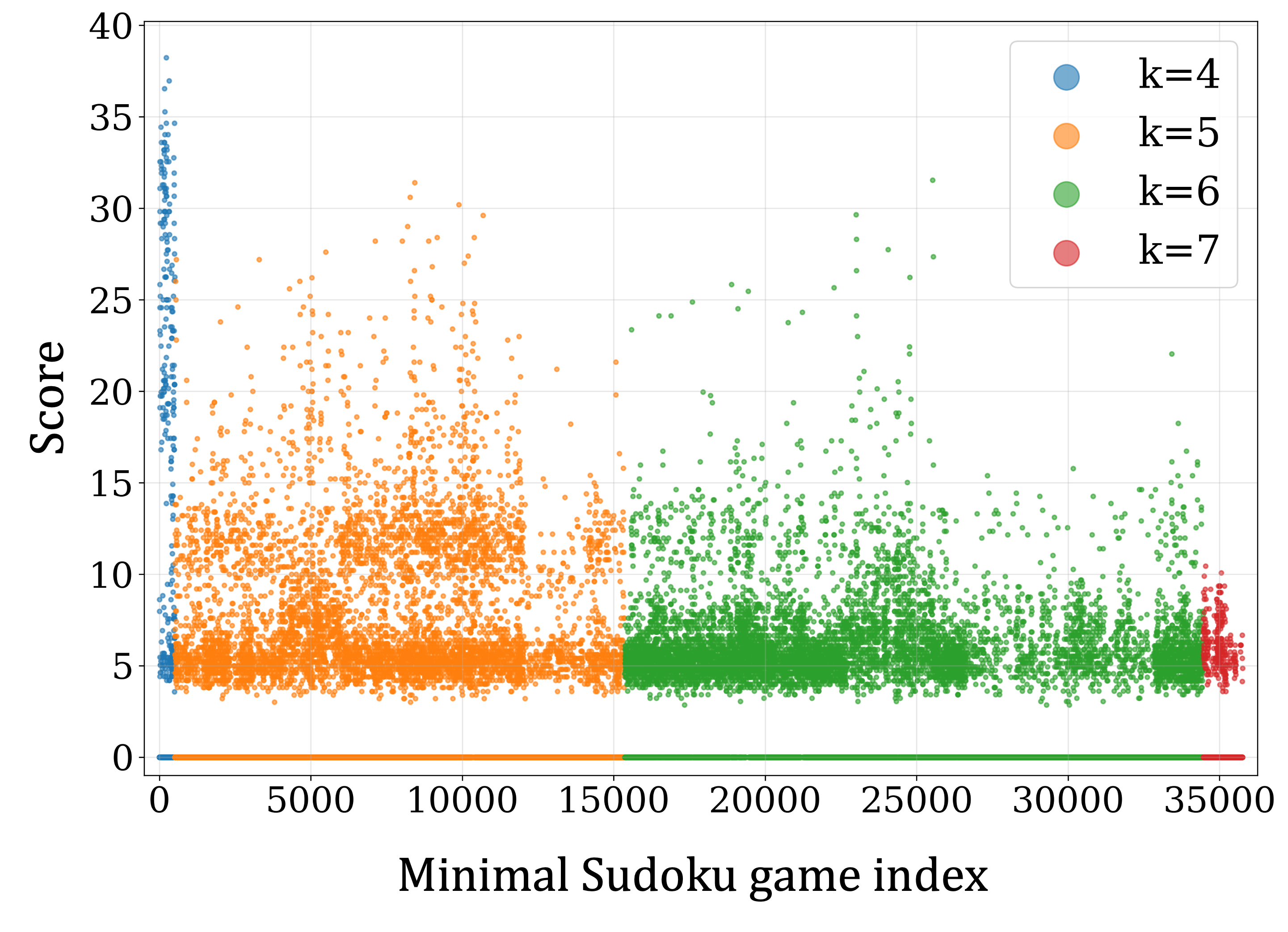}
\caption{Difficulty score of $5\times 5$ minimal Sudoku}
\label{score_scatter}
\end{figure}

In Figure~\ref{score_scatter}, the numbers below each colored dot indicate the index of the corresponding minimal Sudoku, and $k$ denotes the number of given hints in that game.

To provide a more tangible sense of what this score represents, we offer a simple estimation of the solving time. Assuming a player takes, on average, 5 seconds to place each digit, we can calculate the solving time. The total number of digits a player must place is the sum of the number of initially blank cells and the game's difficulty score. For a hypothetical game with $k=5$ hints (20 blank cells) and a difficulty score of 17, the total insertions would be $20 + 17 = 37$. This corresponds to an estimated solving time of $37 \times 5 = 185$ seconds $\approx$ 3 minutes. This represents a reasonable solving time for most players.

This shows that the constructed perfect Sudoku games exhibit a broad distribution in the difficulty score of digit insertions required to solve them. This confirms that our method produces practical games with various difficulties.

The code and minimal Sudoku datasets used in our analysis are available in~\cite{KB25}.

\section{Conclusion}
In this paper, we have investigated the novel connection between coding theory and Sudoku games. We have demonstrated that perfect codes and diameter perfect codes under the Lee distance serve as a systematic framework for constructing symmetric Sudoku-type games. Beyond theoretical construction, our analysis of valid Sudoku grids and their classification under equivalence relations highlights that this framework preserves the high-level symmetries intrinsic to traditional Sudoku and generates a structurally diverse set of valid games. Furthermore, the implementation of a human-like solver has confirmed that our novel Sudoku games are practically playable because they offer a broad spectrum of difficulty levels suitable for actual gameplay.

Future work includes extending our Sudoku-type game construction to codes beyond $\mathbb{Z}_n$ with other distance metrics. Perfect codes exist and have been studied over various rings equipped with other distances including Mannheim distances~\cite{GH14} or rank-metrics~\cite{MZ25}. This observation can provide Sudoku-type games of different sizes and patterns, potentially leading to new games.

\section*{Acknowledgments}
Jon-Lark Kim was supported in part by the BK21 FOUR (Fostering Outstanding Universities for Research) funded by the Ministry of Education (MOE, Korea), National Research Foundation of Korea (NRF) under Grant No. 4120240415042, Basic Science Research Program through the National Research Foundation of Korea (NRF) funded by the Ministry of Science and ICT under Grant No. RS-2025-24534992, and Global - Learning \& Academic research institution for Master’s·PhD students, and Postdocs(LAMP) Program of the National Research Foundation of Korea(NRF) grant funded by the Ministry of Education(No. RS-2024-00441954).

\end{document}